\documentclass[11pt]{amsart}
\usepackage{amsmath,amssymb,amsfonts}
\usepackage{enumerate}
\usepackage{amscd, amssymb, latexsym, amsmath, amscd}
\usepackage[all]{xy}
\usepackage{pb-diagram}
\usepackage{mathrsfs}

\usepackage{pb-diagram}
\theoremstyle{plain}
\newtheorem{Thm}[equation]{Theorem}
\newtheorem{Conj}[equation]{Conjecture}
\newtheorem{Cor}[equation]{Corollary}
\newtheorem{Prop}[equation]{Proposition}
\newtheorem{Lem}[equation]{Lemma}
\numberwithin{equation}{section}

\theoremstyle{remark}
\newtheorem{Rmk}[equation]{Remark}

\newcommand{\Hom}{\operatorname{Hom}}
\newcommand{\Gal}{\operatorname{Gal}}
\newcommand{\GL}{\operatorname{GL}}

\newcommand{\SO}{\operatorname{SO}}
\newcommand{\Sp}{\operatorname{Sp}}
\newcommand{\GSp}{\operatorname{GSp}}
\newcommand{\GSO}{\operatorname{GSO}}

\newcommand{\Mp}{\operatorname{Mp}}

\newcommand{\U}{\operatorname{U}}
\newcommand{\Vol}{\operatorname{Vol}}
\newcommand{\BC}{\operatorname{BC}} 

\newcommand{\C}{\mathbb C}
\newcommand{\A}{\mathbb{A}}

\newcommand{\Z}{\mathbb{Z}}

\newcommand{\bm}{\begin{multline*}}
\newcommand{\tu}{\end  {multline*}}

\newcommand{\disc}{{\rm disc}}

\title{Unitary Friedberg-Jacquet Periods}
\author{Rui Chen and Wee Teck Gan} 
\address{Department of Mathematics, National University of Singapore, Block S17, 10 Lower Kent Ridge Road, Singapore 119076. }
\email{matgwt@nus.edu.sg}
\email{e0046839@u.nus.edu}
\dedicatory{to Steve Kudla, \\ in admiration and appreciation\\ on the occasion of his 70th birthday}
\begin{document}
\maketitle
 
\section{\bf Introduction}
A much studied branching problem in the representation theory of reductive groups over local fields $F$ is the so-called linear period or Friedberg-Jacquet period for $\GL_{2n}(F)$. This concerns the subgroup 
\[  H = \GL_n(F) \times \GL_n(F) \subset \GL_{2n}(F), \]
which is a Levi subgroup of the maximal parabolic subgroup $P = H \cdot N$ of type $(n,n)$.  For an irreducible smooth representation $\pi$ of $\GL_{2n}(F)$, one is interested in determining $\dim \Hom_H(\pi, \C)$.  On the other hand, this linear period is  closely related to the so-called Shalika period. The Shalika period concerns the subgroup 
\[  S =  \GL_n(F)^{\Delta} \cdot N \subset P:= H \cdot N  \]
where $\GL_n(F)^{\Delta} \hookrightarrow H$ is the diagonal embedding. In this case, one takes a generic character $\psi$ of $N$ which is fixed by $\GL_n(F)^{\Delta}$ and would like to determine $\dim \Hom_S(\pi, 1 \otimes \psi)$.

The following theorem summarizes the main result, which is the culmination of the work of many people (for example \cite{CS, G,  JNQ, JR, K, KR, LM, M1, M2}).
\vskip 5pt

\begin{Thm}  \label{T:split-local}
For generic $\pi \in {\rm Irr}(\GL_{2n}(F))$,  one has
\[  \dim \Hom_H(\pi, \C)  =  \dim \Hom_S(\pi, 1 \otimes \psi) \leq 1. \]
Moreover, the above dimensions are nonzero if and only if the L-parameter of $\pi$ is symplectic. 
When $\pi$ is a discrete series representation, this last condition is equivalent to its exterior square L-factor $L(s, \pi, \wedge^2)$ having a pole at $s = 0$. Here, the exterior square L-function can be equivalently defined in terms of the L-parameter of $\pi$, by the Langlands-Shahidi method or the local zeta  integrals of Jacquet-Shalika \cite{JS}.
\end{Thm}

\vskip 5pt
One can also study the analogous period problem in the global setting, where one considers the Friedberg-Jacquet or Shalika period integrals of cuspidal representations of 
$\GL_{2n}$ over a number field $k$. There is an analogous global theorem:
\vskip 5pt

\begin{Thm} \label{T:split-global}
Let $\Pi$ be an irreducible  unitary cuspidal representation of $\GL_{2n}$ over a number field $k$.
\vskip 5pt

(i) $\Pi$ has nonzero global Shalika period if and only if its (partial) exterior square L-function $L(s, \Pi, \wedge^2)$ has a pole at $s  =1$.
\vskip 5pt

(ii) $\Pi$ has nonzero Friedberg-Jacquet period if and only if $L(\Pi, 1/2) \ne 0$ and $L(s, \Pi, \wedge^2)$ has a pole at $s  =1$.
\end{Thm}
These global results are consequences of \cite{JS} for (i) and \cite{BF} for (ii). What is interesting here is the appearance of the extra condition on the nonvanishing of the standard L-function at $s = 1/2$ in (ii).
\vskip 5pt

Variants of the above periods have been considered. For example, instead of taking the trivial character of $H$ or $\GL_n(F)^{\Delta}$, one could take any unitary character for the above periods. Further, Jacquet has initiated the study of  the $\GL_n(E)$-period in $\GL_{2n}(F)$ or its inner forms, where $E/F$ is a quadratic \'etale algebra; when $E = F \times F$, one recovers the linear period above.   Freidberg-Jacquet \cite{FJ} and Guo \cite{G1, G2} have proposed a relative trace formula (RTF) approach to classify representations of $\GL_{2n}(F)$ or its inner forms  distinguished by $\GL_n(E)$.  There has been continual work on these periods and RTF's in the past 25 years. The same is true for the Shalika periods on inner forms of $\GL_{2n}$, as exemplified by the recent paper of Beuzart-Plessis and Wan \cite{BW1}, which introduces a local trace formula approach to study such problems.
\vskip 10pt

On the other hand, one can consider  analogs of the above periods in the setting of unitary groups. More precisely, if $\U_{2n}(F)$ is a unitary group associated to a rank $2n$ Hermitian or skew-Hermitian space , one may consider
\vskip 5pt

\begin{itemize}
\item[(a)] unitary Friedberg-Jacuqet (FJ) periods with respect to subgroups $\U_n \times \U_n \subset \U_{2n}$;
\vskip 5pt

\item[(b)] unitary Shalika periods with respect to a Shalika subgroup $S = \U_n^{\Delta} \cdot N$  when $\U_{2n}$ is quasi-split.
\end{itemize}

\noindent We will provide more precise definitions of these in \S \ref{SS:USP} and \S \ref{SS:UFJ} below.
In the global setting, the paper \cite{GW} of Getz-Wambach formulates a general principle relating these unitary versions to the  $\GL$ case considered earlier, from the viewpoint of the relative trace formula (RTF).   This RTF approach for the unitary FJ periods is being pursued 
in the ongoing work of Jingwei Xiao and Wei Zhang, which attempts to compare the RTF's on unitary groups and general linear groups, and that of Spencer Leslie \cite{L1, L2} which attempts to develop a theory of endoscopy for the RTF in question. In addition, the recent paper \cite{PWZ} of Pollack-Wan-Zydor considers these period problems through the lens of  the residue method pioneered by Ginzburg-Rallis-Soudry and Jiang.   In the local nonarchimedean setting, there is the recent work \cite{BW2, BW3} of Beuzart-Plessis and Wan which extends \cite{BW1} and determines precisely the multiplicity of unitary Shalika periods  for discrete series representations of the quasi-split $\U_{2n}$, via a local trace formula.  
\vskip 10pt

With this background information, the main goal of this paper is to study the unitary FJ period through their connection with the unitary Shalika period via theta correspondence.
This connection via theta correspondence has been expounded in the paper \cite{G} of the second author. Locally, this allows us to determine 
\[  \dim \Hom_{\U_n \times \U_n}(\pi,  \chi_1 \otimes \chi_2) \]
for any irreducible discrete series representation of $\U_{2n}(F)$ and arbitrary characters $\chi_1$ and $\chi_2$ of $\U_n(F)$. In addition, we shall determine
\[  \dim \Hom_{\GL_n(E)}(\pi,  \chi)  \]
for any irreducible discrete series representation $\pi$ of $\U_{2n}(F)$ and  any unitary character $\chi$ of $\GL_n(E)$. 
This gives the unitary version of Theorem \ref{T:split-local} for discrete series representations.
Globally, we shall  prove one direction of a conjecture of Xiao-Zhang, stating that the nonvanishing of a global unitary FJ-period implies the nonvanishing of the central value of the (twisted) standard L-function. This is a partial analog of Theorem \ref{T:split-global}(ii).  Indeed, our method allows us to treat the general case of $(\U_a \times \U_b)$-period
in $\U_n$, with $a+b = n$, though we only highlight the case when $|a- b| \leq 1$ in this paper. 
\vskip 5pt

We shall leave the precise formulation of the theorems to the main body of the article, and conclude this introduction with a summary of the subsequent sections: 

\vskip 10pt

\begin{itemize}
\item In \S \ref{S:prelim}, we begin with introducing some basic preliminaries and notations for $\epsilon$-Hermitian forms, their isometry groups and the Siegel parabolic subgroup, before giving the precise definitions of the unitary FJ and Shalika periods. 

\vskip 5pt

\item In \S \ref{S:theta}, we explicate the precise relation between the two using local theta correspondence. 
This relation has been shown in the paper \cite{G} of the second author, for both the symplectic-orthogonal and unitary dual pairs. However, there were some unfortunate misprints in various formulas in that paper which render the final result (in Theorem 2.5) inaccurate in the unitary case (the formula remains valid in the symplectic-orthogonal case).  Because of this, we give a retreatment of \cite[Thm. 2.5]{G} here. The main results are Theorem \ref{T:SFJ} and Corollary \ref{C:SFJ}.
\vskip 5pt

\item In \S \ref{S:USP}, we recall the results of Beuzart-Plessis and Wan on the multiplicities of unitary Shalika period for discrete series representations of quasi-split $\U_{2n}$.
Their results serve as  the crucial input for this paper.  
\vskip 5pt

\item In \S \ref{S:UFJ}, we combine the results of the previous two sections to determine the multiplicities of unitary FJ periods for discrete series representations of arbitrary $\U_{2n}$. The main results are given in Theorem \ref{T:UFJ-1} and Theorem \ref{T:UFJ-2}. Their variants for the $(\U_n \times\U_{n+1})$-periods in $\U_{2n+1}$ are contained in Theorem \ref{T:UFJ-I} and Theorem \ref{T:UFJ-II}.

\vskip 5pt

\item In \S \ref{S:variant}, we consider yet another variant of the FJ period, namely the $\GL_n(E)$-period for representations of quasi-split $\U_{2n}(F)$. By an internal analysis, we 
show a precise result relating this period to the unitary Shalika period. This thus allows us to determine the multiplicities of this period for discrete series representations.
\vskip 5pt
 
\item  In \S \ref{S:global}, we consider the global setting and recall a conjecture of Xiao-Zhang in the spirit of Theorem \ref{T:split-global}(ii). After establishing the global analog of the relation between unitary FJ and Shalika periods, we prove one direction of their conjecture in Proposition \ref{P:XZ-1} and Corollary \ref{C:XZW}. The analog for the global $(\U_n \times\U_{n+1})$-periods in $\U_{2n+1}$ is contained in Proposition \ref{P:XZ-2}.
\vskip 5pt

\item  Finally, in \S \ref{S:final}, we revisit the Rankin Selberg-integrals of Jacquet-Shalika and Bump-Friedberg which are responsible for giving Theorem \ref{T:split-global}. These Rankin-Selberg integrals suggest two new branching problems, with the mirabolic Eisenstein series replaced by the Weil representation of unitary groups. This is similar in spirit to the twisted GGP conjecture recently formulated in \cite{GGP}. We hesitate to make precise conjectures for these branching problems here, but merely want to suggest that they are natural problems which should have nice answers.
\end{itemize}
\vskip 15pt

 \section{\bf Preliminaries}  \label{S:prelim}
 
 In this section, we introduce the basic setup before defining the unitary Shalika period and Friedberg-Jacquet periods. 
\vskip 5pt

\subsection{\bf Fields} \label{SS:fields}
Let $F$ be a field of characteristic not $2$ and let $E$ be a  (separable) quadratic field extension of $F$ with trace map $Tr_{E/F}$ and norm map $N_{E/F}$. 
Let 
\[  E_0 = \{ e \in E: Tr_{E/F}(e) =0\} \quad \text{and} \quad  E_1 = \{ e \in E:  N_{E/F}(e) =1\}. \]
Then $E_0$ is a 1-dimensional $F$-vector space and $E_1$ is a one-dimensional anisotropic torus.
There is a natural isomorphism
\[  i : E^{\times}/ F^{\times}  \longrightarrow  E_1 \quad \text{given by $i(e) = e / e^c $,}  \]
where $c: e \mapsto e^c$ is the nontrivial element of ${\rm Gal}(E/F)$.
 
\vskip 5pt

\subsection{\bf Conjugate spaces}
Let $X$ be a finite-dimensional $E$-vector space. It will be necessary for us to consider the conjugate vector space
\[  X^c :=  X \otimes_{E,c} E. \]
The elements of $X^c$ are linear combinations of $x \otimes e$ with $x \in X$ and $e \in E$, so that $x \otimes e = e^cx \otimes 1$.
An element $g \in \GL(X)$ acts naturally and $E$-linearly on $X^c$ by $g \cdot (x \otimes 1) = (g \cdot x) \otimes 1$. 
This gives a natural map
\[  \iota: \GL(X) \longrightarrow \GL(X^c). \]
Note however that
\begin{equation} \label{E:det}
 {\det}_{X^c} \circ \iota  = {\det}_X^c. \end{equation}
\vskip 5pt

\subsection{\bf $\epsilon$-Hermitian forms} \label{SS:forms}
A nondegenerate  $\epsilon$-Hermitian form on $X$ is a nondegenerate $F$-bilinear pairing
\[  \langle- , -\rangle_X : X \otimes X \longrightarrow E \]
such that
\vskip 5pt

\begin{itemize}
\item $\langle -, -\rangle$ is $E$-linear in the first variable and $E$-conjugate-linear in the second;

\item for $x_1, x_2 \in X$,
\[  \langle x_1, x_2 \rangle_X^c  = \epsilon \cdot \langle x_2 , x_1 \rangle_X . \]
 \end{itemize}
 When $\epsilon =1$, the form is Hermitian, whereas if $\epsilon = -1$, it is skew-Hermitian.
 \vskip 5pt
 
 Now we may regard $\langle- , - \rangle_X$ as an $E$-bilinear pairing $X \times X^c \rightarrow E$, via 
  $(x_1, x_2 \otimes 1) \mapsto \langle x_1, x_2 \rangle_X$. 
 Since  $X$ and $X^c$ play symmetrical roles,  we can naturally  regard $\langle-, -\rangle_X$  as an $\epsilon$-Hermitian form on $X^c$:
 \[  \langle x_1\otimes 1, x_2 \otimes 1 \rangle_{X^c} = \langle x_2, x_1 \rangle_X = \epsilon \cdot \langle x_1, x_2 \rangle_X^c. \]
We shall make such natural identifications below without too much further comment.
\vskip 5pt

Let $\U(X) \subset \GL(X)$ denote the isometry group of $\langle-,-\rangle_X$.  For $g \in \U(X)$, it is easy to verify that $\iota(g) \in \GL(X^c)$ preserves the naturally associated $\langle-,-\rangle_{X^c}$. In other words, the natural map $\iota: \GL(X) \rightarrow \GL(X^c)$ restricts to give a natural isomorphism $\U(X) \cong \U(X^{c})$.

\vskip 5pt

 \subsection{\bf Siegel parabolic} \label{SS:siegel}
 Assume that  $W$ is a skew-Hermitian space of even dimension $n$ which is maximally split, in the sense that its maximal isotropic subspaces are of dimension $n/2$. 
 Then we have a Witt decomposition
\[  W = X \oplus  Y, \quad \text{with $X$ and $Y$ maximal isotropic.} \]
The skew-Hermitian form $\langle-,-\rangle_W$ gives an identification of $Y$ with the $E$-linear dual $\Hom_E(X,E)$ of $X$ via:
\[  y :  x \mapsto \langle x, y \rangle_W. \]
However, the map $Y \longrightarrow \Hom_E(X,E)$ is conjugate-$E$-linear, rather than $E$-linear. In other words,  we have an $E$-linear isomorphism
\[  Y^c \cong   \Hom_E(X,E). \]
Equivalently, $X \cong \Hom_E(Y^c, E)$. 
\vskip 5pt

 The stabilizer $P(X)$ of $X$ in $\U(W)$ is a maximal parabolic subgroup known as a Siegel parabolic. Its unipotent radical $N(X)$ is abelian and can be described as follows.
 For $A \in \Hom_E(Y,X)$, let $n(A) \in \GL(W)$ be the element   which is the identity map on $X$ and $n(A)(y) = y + A(y)$ for $y\in Y$. Then $n(A)$ preserves $\langle-,-\rangle_W$ if and only if 
 \begin{equation} \label{E:N}
   \langle Ay_1, y_2 \rangle_W  + \langle y_1, Ay_2 \rangle_W = 0 \quad \text{for all $y_1, y_2 \in Y$.} \end{equation}
 Then
\[  N(X) \cong  \{ n(A): A  \in \Hom_E(Y,X) \, \text{and (\ref{E:N}) holds}   \}  \hookrightarrow \Hom_E(Y,X).  \]

Now an element $n(A) \in N(X)$ can be regarded as a Hermitian form on $Y$, via
\[  A_Y (y_1, y_2) := \langle Ay_1, y_2 \rangle. \]
As we mentioned in the previous subsection, we may also regard $A$ as a Hermitian form on $Y^c$. Denoting this Hermitian form by $A_{Y^c}$, we note that
\[  A_{Y^c}(y_1 \otimes 1, y_2 \otimes 1)  = A_Y(y_2, y_1) = \langle Ay_2, y_1 \rangle_W.  \]
In the following, we shall use $A$ to denote the Hermitian form on $Y$ and $A^c$ to denote the corresponding Hermitian form  on $Y^c$.
\vskip 5pt

The subgroup of $P(X)$ stabilizing the Witt decomposition $W = X \oplus Y$ is a Levi subgroup $M_{X,Y}$. By definition, it is equipped with natural isomorphisms
\[  M_{X,Y} \longrightarrow \GL(X) \quad \text{and} \quad M_{X,Y} \longrightarrow \GL(Y). \]
We shall identify $M_{X,Y}$ with $\GL(X)$ and thus inherits an identification $\dagger: \GL(X) \longrightarrow \GL(Y)$ characterized by
\[   \langle g \cdot x,  g^{\dagger} \cdot y \rangle_W = \langle x,  y \rangle_W. \]
Recalling that $Y \cong \Hom_E(X^c, E)$, the action of $\GL(X)$ on $Y$ through $\dagger$ is nothing but the natural action of $\GL(X)$ on $\Hom_E(X^c, E)$ via $\iota$ and duality. 
\vskip 5pt

Applying the above discussion to the Siegel parabolic $P(Y)$ stabilizing $Y$, we see that the elements $n(B)$ of its unipotent radical $N(Y) \subset \Hom_E(X,Y)$ can be regarded as Hermitian forms  on $X$ or $X^c$, denoting these by $B$ and $B^c$ respectively. As explained above, the isomorphism $\iota: \GL(X) \rightarrow \GL(X^c)$ restricts to an isomorphism $\U(X,B) \cong \U(X^c, B^c)$.

\vskip 5pt
 \subsection{\bf Non-archimedean fields} \label{SS:nonarch}
We suppose henceforth that $F$ and $E$ are nonarchimedean local fields and let $\omega_{E/F}$ denote the quadratic character of $F^{\times}$ associated to $E/F$ by local class field theory. We also fix a non-zero element $\delta\in E_0$. Throughout this paper, we will let $V$ denote a finite-dimensional Hermitian space with form $\langle-, -\rangle_V$ and $W$ a skew-Hermitian space over $E$ with form $\langle-,-\rangle_W$. Let us briefly recall the classification of such $\epsilon$-Hermitian spaces over $E$.
\vskip 5pt

For each given dimension $n$, there are precisely two distinct Hermitian spaces $V$, distinguished by the sign
\[
	\epsilon(V)=\omega_{E/F}(\disc V),
\]
where
\[  
	\disc(V) = (-1)^{n(n-1)/2} \cdot \det(V) \in F^{\times}/ N_{E/F}(E^{\times}). 
\]
Likewise, there are two distinct skew-Hermitian spaces $W$, distinguished by the sign
\[
	\epsilon(W)=\omega_{E/F}(\delta^{-n}\cdot\disc W),
\] 
where
\[  
	\disc(W) = (-1)^{n(n-1)/2} \cdot \det(W) \in \delta^n\cdot F^{\times}/ N_{E/F}(E^{\times}). 
\]
Notice that the sign of an odd skew-Hermitian space depends on the choice of $\delta$. Here, we are following the setup of \cite{GI2}.

\vskip 5pt

When $n$ is odd, both the Hermitian (or skew-Hermitian) spaces are maximally split, in the sense that their maximal isotropic subspaces have dimension $(n-1)/2$. When $n$ is even, 
exactly one of the two Hermitian (resp. skew-Hermitian) spaces is maximally split, namely the one with trivial sign.  
Let $\U(V)$ and $\U(W)$ denote the isometry groups of $V$ and $W$ respectively. These isometry groups are quasi-split precisely when the underlying space is maximally split. 

\vskip 5pt

 \subsection{\bf Unitary Shalika  Periods}  \label{SS:USP}

We are now ready to introduce the unitary Shalika   period.
Assume  
\[  W = X \oplus Y  \]
with $X$ and $Y$ maximal isotropic and consider the   maximal (Siegel) parabolic subgroup
\[  P(X)  = \GL(X)  \cdot N(X)  \]
stabilizing $X$. 
 
 Fix a nontrivial additive character $\psi$ of $F$. Using $\psi$,
we may identify the Pontryagin dual of $N(X)$ with the unipotent radical $N(Y) \subset \Hom(X, Y)$  of the opposite parabolic $P(Y)$.  More precisely,  the element $n(B) \in N(Y)$ gives the character
\[  \psi_B (n(A))  = \psi\left(\frac{1}{2}\,Tr_X(A \circ B)\right)  \quad \text{  for $n(A) \in N(X)$.}  \]
 
 \vskip 5pt

For a given $B \in N(Y)$, the stabilizer of $\psi_B$ in $\GL(X)$ is the isometry group $\U(X,B) \cong  U(X^c, B^c)$ (isomorphism given by $\iota$). The subgroup
\[  S_B = \U(X,B) \ltimes N(X) \subset P(X) \]
is the Shalika subgroup relative to $B$. For any character $\mu \circ \det_X$ of $\U(X,B)$ (where $\mu$ is a character of $E_1$), $(\mu\circ \det_X) \boxtimes \psi_B$ defines a character of $S_B$ and the $(B, \mu)$-Shalika period of $\pi \in {\rm Irr}(\U(W))$ is the space
\[  
	Sha_{B}(\pi, \mu):= \Hom_{S_B}(\pi, \mu \circ {\det}_X \boxtimes \psi_B). 
\]
By duality and Frobenius reciprocity,
\[  Sha_B(\pi,\mu) \cong \Hom_{\U(W)} \left(  {\rm ind}_{S_B}^{\U(W)} (\mu^{-1} \circ {\det}_X\boxtimes \psi_B^{-1}), \pi^{\vee}\right). \] 
We shall set
\[   Sha(B,\mu) := {\rm ind}_{S_B}^{\U(W)} (\mu \circ  {\det}_X  \boxtimes \psi_B)  \]
and call it the Shalika module of $\U(W)$ relative to $(B,\mu)$. 
\vskip 5pt

Now there is of course no reason to insist that $\mu$ is a 1-dimensional character in the above. We could have replaced $\mu$ by any smooth representation $\sigma$ of $\U(X,B)$ and 
thus have the period space $Sha_B(\pi, \sigma)$ and the $\U(W)$-module $Sha(B,\sigma)$.

\vskip 5pt
\subsection{\bf Unitary Friedberg-Jacquet period}  \label{SS:UFJ}
Now suppose
\[   V = V_0 \oplus V_0^{\perp}, \]
where $V_0$ is a nondegenerate Hermitian subspace of $V$. Then 
\[  H_{V_0} := \U(V_0) \times \U(V_0^{\perp})  \subset \U(V). \]
For two characters $\mu_1$ and $\mu_2$ of $E_1$, the $(\mu_1, \mu_2)$-linear period of $\pi \in {\rm Irr}(\U(V))$ is the space
\[  Lin_{V_0}(\pi, \mu_1 \boxtimes \mu_2):= \Hom_{H_{V_0}}(\pi, (\mu_1 \circ {\det}_{V_0})  \boxtimes (\mu_2  \circ {\det}_{V_0^{\perp}})). \]
We also call these the unitary Friedberg-Jacquet periods.
\vskip 5pt

As in the case of (unitary) Shalika periods, we also set
\[  
	Lin(V_0, \sigma_1 \boxtimes \sigma_2) := {\rm ind}_{\U(V_0) \times \U(V_0^{\perp})}^{\U(V)} \sigma_1\boxtimes \sigma_2 
\]
for smooth representations $\sigma_1$ and $\sigma_2$ of $\U(V_0)$ and $\U(V_0^{\perp})$ respectively. We shall call this the Friedburg-Jacquet module of $\U(V)$ relative to $(\sigma_1, \sigma_2)$.
\vskip 15pt

\section{\bf Theta Correspondence and Periods}  \label{S:theta}
This section serves as an erratum to \cite[Thm. 2.5]{G}, which relates the unitary Shalika period with the unitary Friedberg-Jacquet period via theta correspondence. This is necessitated by the occurrence of several unfortunate misprints in \cite[Prop. 2.1 and Thm. 2.5]{G}.
\vskip 5pt

\subsection{\bf Theta correspondence} \label{SS:theta-splitting}
Let $W$ be a skew-Hermitian space  and $V$ a Hermitian space over $E$, with isometry groups $\U(W)$ and $\U(V)$ respectively. Then the group $\U(W) \times \U(V)$ has an associated Weil  representation $\Omega_{\psi}$ which depends also on a pair of characters $(\chi_V, \chi_W)$ of $E^{\times}$ (see \cite[\S 4]{GI} for details). This pair of characters satisfies:
\[  \chi_V|_{F^{\times}} = \omega_{E/F}^{\dim V} \quad \text{and} \quad \chi_W|_{F^{\times}} = \omega_{E/F}^{\dim W}. \]
Suppose that $W = X \oplus Y$ is a Witt decomposition (so that $\dim W$ is maximally split and even-dimensional).
We will work with the Schr\"odinger model of the Weil representation $\Omega$ relative to the Siegel parabolic $P(X)$. This is realized on the space $\mathcal{S}(Y \otimes V)$ of Schwarz functions on $Y \otimes V = \Hom_E(X^c, V)$. The action of $\U(V)\times P(X)$ is given as follows:
\vskip 5pt

\begin{align} 
&(h \cdot \phi)( T)  = \chi_W(i^{-1}({ \det}_V(h))) \cdot  \phi(h^{-1} \circ T) &\quad \text{  for $h \in \U(V)$;} \label{F:Schro-1}\\
& (m \cdot \phi) (T)  = \chi_V({\det}_X(m)) \cdot  |{\det}_X(m)|^{\frac{1}{2} \dim V}  \phi(T \circ m) &\quad \text{  for $m \in \GL(X)$;} \label{F:Schro-2}\\
& (n(A)  \cdot \phi)(T)  = \psi_{T^*(V)} (A)  \cdot \phi(T)  &\quad \text{  for $n(A) \in N(X)$.} \label{F:Schro-3} 
\end{align}
Here  $T^*(V)$ is the Hermitian form on $X^c$ obtained by pulling back the Hermitian form on $V$ using $T \in \Hom_E(X^c, V)$; as we have noted, $T^*(V)$ can be naturally regarded as Hermitian form on $X$ and as an element in $\Hom_E(X,Y)$.
 \vskip 5pt
 
 It may be pertinent to point out the misprints in \cite[\S 2]{G}:
 \begin{itemize}
 \item for the action of $h \in \U(V)$ in (3.1), the RHS was written as $\chi_W ({\det}_V(h))$ in \cite{G}, i.e. $i^{-1}$ was missing there.
 \item the space $Y \otimes V$ (denoted by $X^* \otimes V$ in \cite{G}) was identified with $\Hom_E(X, E)$ in \cite{G}, instead of $\Hom_E(X^c, V)$; in other words, $X$ should be $X^c$ at various places in \cite[\S 2]{G}.
 \end{itemize}
 Note that these issues only affect the case of unitary dual pairs, so that \cite[Thm. 2.5]{G} is correct for symplectic-orthogonal dual pairs.

\vskip 5pt

For an irreducible  representation $\pi$ of $\U(V)$, we have its big theta lift $\Theta(\pi)$ which is a smooth representation of $\U(W)$, so that there is a functorial isomorphism
\[  \Hom_{\U(V)\times\U(W)}( \omega_{\psi},   \pi \boxtimes \Sigma)  \cong \Hom_{\U(W)}(\Theta(\pi),    \Sigma)  \]
for any smooth representation $\Sigma$ of $\U(W)$.  
\vskip 5pt

\subsection{\bf Transfer of periods} \label{SS:TOP}
The purpose of this subsection is to determine
\[  \Theta(\pi)_{N(X), \psi_B}^*  = \Hom_{N(X)}(\Theta(\pi), \psi_B) =   \Hom_{\U(V) \times N(X)}(\Omega, \pi \boxtimes \psi_B)  \]
for $\pi \in {\rm Irr}(\U(V))$ and $B \in N(Y)$ a nondegenerate Hermitian form on $X$. We first note that
\[ 
 \Theta(\pi)_{N(X), \psi_B}^*    = \Hom_{\U(V)}( \Omega_{N(X), \psi_B}, \pi). \]
The following proposition determines $\Omega_{N(X), \psi_B}$.
\vskip 5pt

\begin{Prop}
As a representation of $\U(V)$,  $\Omega_{N(X), \psi_B} = 0$ if there is no embedding of Hermitian spaces $B^c  \hookrightarrow   V$. If there is an embedding $j : B^c \hookrightarrow V$ of Hermitian spaces, then we may write $V = j(X^c)  \oplus j(X^c)^{\perp}$ and
\begin{align}
  \Omega_{N(X), \psi_B} & \cong  {\rm ind}_{\U(j(X^c)^{\perp})}^{\U(V)}  \left( \chi_W \circ i^{-1} \circ {\det}_{j(X^c)^{\perp}}  \right)\notag \\
  &\cong (\chi_W \circ i^{-1} \circ {\det}_V) \cdot  {\rm ind}_{\U(j(X^c)^{\perp})}^{\U(V)} 1. \notag
  \end{align}
\end{Prop}

\begin{proof}
 
From the formulas for the Schr\"odinger model of $\Omega$ given in (3.1)-(3.3), one deduces that there is a natural isomorphism of $\U(V)$-modules:
\[  \Omega_{N(X), \psi_B} \cong  \mathcal{S} ( \mathcal{O}_B) , \]
where $\mathcal{O}_B \subset \Hom_E(X^c,V)$ is the Zariski closed subset
\[  \mathcal{O}_B = \{  T  :   T^*(V)   = B^c \} = \{ \text{embeddings of Hermitian spaces $B^c \hookrightarrow V$}\}. \]
Here, note that $\U(V)$ preserves $\mathcal{O}_B$ and its action on $\mathcal{S}(\mathcal{O}_B)$ is geometric.
 The  natural projection $\Omega \rightarrow \Omega_{N(X),\psi_B}$ is simply given by the restriction of functions from $Y \otimes V$ to the Zariski closed subset $\mathcal{O}_B$.  
\vskip 5pt

Hence, if $\mathcal{O}_B$ is empty, i.e. if there is no embedding $B^c \hookrightarrow V$, then $\Omega_{N(X), \psi_B} =0$, as desired.  On the other hand, if $\mathcal{O}_B$ is nonempty, then Witt's theorem says that $\U(V)$ acts transitively on $\mathcal{O}_B$. If we fix a base point $j \in \mathcal{O}_B$, and write  $V = j(X^c)  \oplus j(X^c)^{\perp}$, then the stabilizer of $j$ in $\U(V)$ is equal to $\U(j(X^c)^{\perp})$. As a consequence, we deduce that
\[  \Omega_{N(X), \psi_B}  \cong  {\rm ind}_{\U(j(X^c)^{\perp})}^{\U(V)}  \chi_W \circ i^{-1} \circ {\det}_{j(X^c)^{\perp}} \]
as desired.
\end{proof}

\vskip 5pt

From the proposition, we see that 
\begin{align}
   \Hom_{N(X)}(\Theta(\pi), \psi_B)  &\cong \Hom_{\U(V)} ( {\rm ind}_{\U(j(B^c)^{\perp})}^{\U(V)}  \chi_W \circ  i^{-1} \circ {\det}_V , \pi)  \notag \\
   &\cong \Hom_{\U(j(B^c)^{\perp})}(\pi^{\vee}, \chi_W^{-1} \circ i^{-1} \circ {\det}_V) \notag
   \end{align}
as vector spaces. Here to get the second isomorphism, we have made use of duality and Frobenius reciprocity.
\vskip 5pt

 In fact,  there are some extra symmetries here. More precisely, the stabilizer in $\GL(X)$ of the character $\psi_B$ is the subgroup $\U(B)$ and 
$\Omega_{N(X),\psi_B}$ is naturally a representation of $\U(B) \times \U(V)$.  Now the embedding $j: X^c \hookrightarrow V$ induces an isomoprhism
\[   \begin{CD} 
 \U(B)  @>\iota>> \U(B^c) @>{\rm Ad}(j)>>  \U(j(B^c))   \end{CD} \]
 where ${\rm Ad}(j)(h) = j \circ h \circ j^{-1}$. This gives a natural diagonal embedding
 \[    \begin{CD}
\Delta: \U(j(B^c))    @>({\rm Ad}(j) \circ \iota)^{-1} \times {\rm id}>> \U(B) \times \U(j(B^c))   @>>> \U(B)  \times \U(V). \end{CD} \]
 Keeping track of the $\U(B)$-action in the proof of the above proposition, one sees that as a $\U(B) \times \U(V)$-module, 
\[  \Omega_{N(X) ,\psi_B} \cong  
    {\rm ind}_{\U(j(B^c))^{\Delta}  \times \U(j(B^c)^{\perp})}^{\U(B) \times \U(V)}    (\chi_V^{-2} \chi_W \circ i^{-1} \circ {\det}_{j(B^c)} )\boxtimes  (\chi_W \circ i^{-1} \circ  {\det}_{j(B^c)^{\perp}} ).  \]
 Here, we have made use of the facts that
 \[  \chi_V \circ {\det}_X  = \chi_V^2 \circ i^{-1} \circ {\det}_X \quad \text{as characters of $\U(B)$} ,\]
 and  
 \[   (\chi_V \circ {\det}_X ) \circ  ({\rm Ad}(j) \circ \iota)^{-1} = \chi_V^{-1} \circ {\det}_{j(X^c)} \quad \text{ as characters of $\U(j(B^c))$}, \]
 which follows from (\ref{E:det}). 
 Hence we obtain:
\vskip 5pt

\begin{Thm}  \label{T:SFJ}
Fix a nondegenerate Hermitian form $B$ on $X$.
Given $\pi \in {\rm Irr}(\U(V))$ and $\sigma \in {\rm Irr}(\U(B))$, there is an isomorphism of vector spaces
\[
\Hom_{\U(B) \ltimes N(X)}(\Theta(\pi),   \sigma \boxtimes \psi_B)  \cong \]
\[   \Hom_{\U(j(B^c)) \times \U(j(B^c)^{\perp})} \left( \pi^{\vee},  \sigma \circ ({\rm Ad}(j) \circ \iota)^{-1}  \cdot  (\chi_V^2 \chi_W^{-1} \circ i^{-1}  \circ {\det}_{j(X^c)})\boxtimes (\chi_W^{-1} \circ i^{-1} \circ {\det}_{j(X^c)^{\perp}}) \right). \]
\end{Thm}
\vskip 5pt

We will especially be interested in the case when $\sigma = \mu \circ  i^{-1} \circ {\det}_X$ is a character of $\U(B)$, with $\mu$ a character of $E^{\times}/F^{\times}$. When pulled back to a character of $\U(j(B^c))$ via $({\rm Ad(j)} \circ \iota)^{-1}$, we have
\[   \sigma \circ ({\rm Ad}(j) \circ \iota)^{-1} =  \mu^{-1} \circ i^{-1} \circ  {\det}_{j(X^c)}. \]
We record the following corollary of Theorem \ref{T:SFJ}:
\vskip 5pt

\begin{Cor}  \label{C:SFJ}
In the context of Theorem \ref{T:SFJ},
 \[
\Hom_{\U(B) \ltimes N(X)}(\Theta(\pi),   \mu \circ i^{-1} \circ {\det}_X \boxtimes \psi_B)  \cong \]
\[ \Hom_{\U(j(B^c)) \times \U(j(B^c)^{\perp})} \left( \pi^{\vee},   (\mu^{-1} \cdot \chi_V^2 \chi_W^{-1} \circ i^{-1}  \circ {\det}_{j(X^c)})\boxtimes (\chi_W^{-1} \circ i^{-1} \circ {\det}_{j(X^c)^{\perp}}) \right). \]
\end{Cor}

\vskip 5pt

As an example, when $\dim V = \dim W$ is even, we may take $\chi_V = \chi_W = \mu$ to be trivial. Then we see that
\[  \text{$\pi^{\vee}$ has nonzero linear period} \Longleftrightarrow \text{$\Theta(\pi)$ has nonzero Shalika period}. \]

\vskip 10pt

\section{\bf Unitary Shalika Periods} \label{S:USP}
In this section, we recall the results of Beuzart-Plessis and Wan \cite{BW2}, \cite{BW3} on the multiplicities of unitary Shalika period.  
\vskip 5pt

\subsection{\bf Local Langlands correspondence}
In order to state their results, we need the local Langlands correspondence for unitary groups. Here we give a quick review of it. For each given dimension $n$, we denote the two distinct $\epsilon$-Hermition spaces by $V^+$ and $V^-$, according to their signs (see \S \ref{SS:nonarch}). Then there is a finite to one surjective map
\[
	\mathcal{L}: {\rm Irr}(\U(V^+)) \sqcup {\rm Irr}(\U(V^-)) \longrightarrow \Phi(n),
\]
where $\Phi(n)$ is the set of L-parameters for unitary groups of $n$-variables. This map $\mathcal{L}$ preserves various properties such as discreteness and temperedness. 
We recall that an element $\phi \in \Phi(n)$ is an isomorphism class of conjugate-dual representation of the Weil-Deligne group $WD_E$ of $E$ of sign $(-1)^{n-1}$. 
Such a $\phi$ is discrete if and only if it is a multiplicity-free direct sum 
\begin{equation} \label{E:discrete}  \phi = \bigoplus_{i\in I} \phi_i  \end{equation}
of irreducible representations $\phi_i$, such that each $\phi_i$ is conjugate-dual of sign $(-1)^{n-1}$ as well.
For $\phi\in \Phi(n)$, we put $\Pi_\phi=\mathcal{L}^{-1}(\phi)$: this is the local L-packet associated to the L-parameter $\phi$.
\vskip 5pt

  Fix a Whittaker datum $\mathscr{W}$ of $\U(V^+)$. Then relative to this datum, there is a canonical bijection
\[
	\mathcal{J}_{\mathscr{W}}: \Pi_\phi\longrightarrow {\rm Irr}(\mathcal{S}_\phi),
\]
where $\mathcal{S}_\phi$ is the component group associated to the L-parameter $\phi$. 
We shall write $\pi(\phi,\eta)$ for the irreducible representation in $\Pi_\phi$ corresponding to $\eta\in{\rm Irr}(\mathcal{S}_\phi)$. Let $z_\phi$ be the image of $-1\in\GL_n(\mathbb{C})$ in $\mathcal{S}_\phi$. We highlight the property that $\pi(\phi,\eta)$ is a representation of $\U(V^\varepsilon)$ if and only if $\eta(z_\phi)=\varepsilon$.
\vskip 5pt
As an example, consider the case when $\phi$ is a discrete L-parameter with multiplicity-free decomposition as in (\ref{E:discrete}). Then
\[  \mathcal{S}_{\phi} = \prod_{i\in I} \Z/2\Z \cdot a_i. \]
In other words, $\mathcal{S}_{\phi}$ is an elementary abelian 2-group with a canonical basis. The element $z_{\phi}$ is simply the element $\sum_{i \in I} a_i$, i.e. it generates the diagonally embedded $\Z/2\Z$. For an element $a=a_{j_1}+\cdots +a_{j_r}$ in $\mathcal{S}_\phi$, we put
\[
	\phi^a= \phi_{j_1} + \cdots + \phi_{j_r}.
\]
\vskip 5pt

\subsection{\bf Multiplicity of unitary Shalika periods} \label{SS:T-USP}
Now we are able to state the results of Beuzart-Plessis and Wan. We retain the notations in \S \ref{SS:USP}. 
Let $W$ be the maximally split skew-Hermitian space of even dimension $2n$.
Let $\mathscr{W}=\mathscr{W}_\psi$ be the Whittaker datum of $\U(W)$ determined by the $N_{E/F}(E^\times)$-orbit of the additive character $\psi$. We shall use the local Langlands correspondence for $\U(W)$ relative to the Whittaker datum $\mathscr{W}$. For a character $\mu$ of $E_1$, we use $\widetilde{\mu}$ to denote the pull-back of $\mu$ along the map $i:E^\times/F^\times\rightarrow E_1$, i.e. 
\[  \widetilde{\mu} = \mu \circ i. \]
\vskip 5pt

The following two theorems due to Beuzart-Plessis-Wan \cite{BW2}, \cite{BW3} determine the multiplicity of $(B,\mu)$-Shalika periods precisely for discrete series representations.
\begin{Thm}\label{C:USP-1}
Let $\phi$ be a discrete L-parameter of $\U(W)$, and $\mu$ a character of $E_1$. If there exists some $\pi\in\Pi_\phi(\U(W))$ such that $Sha_{B}(\pi, \mu)\neq 0$, then as a representation of $WD_E$, $\phi$ takes value in $\GSp_{2n}(\C)$ with similitude factor $\widetilde{\mu}$, i.e. $\lambda\circ\phi=\widetilde{\mu}$, where
\[
	\lambda: \GSp_{2n}(\C)\longrightarrow \C^\times
\]
is the similitude map.
\end{Thm}
\vskip 5pt

Assume now the discrete L-parameter $\phi$ satisfies the condition above. Given such an L-parameter $\phi$, we may decompose $\phi$ into a ``symplectic part'' and a ``non-symplectic part'', as
\[
	\phi=\sum_{i\in I_\phi}\phi_i + \sum_{j\in J_\phi}\left(\phi_j + \phi_j^\vee\widetilde{\mu}\right),
\]
where $\phi_i$, $\phi_j$ and $\phi_j^\vee$ are all irreducible conjugate symplectic representations of $WD_E$. For each constituent $\phi_i$ in the ``symplectic part'' (indexed by $I_{\phi}$), we require that $\phi_i$ takes value in a similitude symplectic group of the  appropriate dimension with similitude factor $\widetilde{\mu}$. For each constituent $\phi_j$ in the ``non-symplectic part'' (indexed by $J_{\phi}$), we note that $\phi_j\not\cong\phi_j^\vee\widetilde{\mu}$ as a consequence of the discreteness of $\phi$. Let 
\[
	\mathcal{S}_\phi^\Delta = \sum_{i\in I_\phi}\Z/2\Z \, a_i + \sum_{j\in J_\phi}\Z/2\Z \, \left(b_j + b_j^*\right)\subset\mathcal{S}_\phi,
\]
where $a_i$, $b_j$ and $b_j^*$ are the basis elements in $\mathcal{S}_\phi$ corresponding to $\phi_i$, $\phi_j$ and $\phi_j^\vee$ respectively. 

\begin{Thm}\label{C:USP-2}
Suppose that the L-parameter $\phi$ is discrete and takes value in $\GSp_{2n}(\C)$ with similitude factor $\widetilde{\mu}$. Let $\pi=\pi(\phi,\eta)$ be an irreducible discrete series representation of $\U(W)$.
\begin{enumerate}
\item Suppose that $I_\phi\neq\varnothing$. Then $\pi$ has non-zero $(B,\mu)$-Shalika periods if and only if
\[
	\eta~\Big|_{\mathcal{S}_\phi^\Delta}=1.
\] 
When this condition holds, we have
\[
	\dim Sha_{B}(\pi, \mu)=2^{|I_\phi|-1}.
\]
In particular, the dimension of $Sha_{B}(\pi, \mu)$ is independent of the choice of $B$ in this case.
\item Suppose that $I_\phi=\varnothing$. Then $\pi$ has non-zero $(B,\mu)$-Shalika periods if and only if
\[
	\eta~\Big|_{\mathcal{S}_\phi^\Delta}=1, \quad \text{and} \quad \eta\left(\sum_{j\in J_\phi}b_j\right)=\epsilon(B).
\]
Here $\epsilon(B)$ is the sign of the Hermitian space $(X,B)$. When these conditions hold, we have 
\[
	\dim Sha_{B}(\pi, \mu)=1.
\]
\end{enumerate}
\end{Thm}

\vskip 10pt

\section{\bf Unitary Friedberg-Jacquet Periods} \label{S:UFJ}
In this section, we consider the unitary FJ period. We retain the notations in \S \ref{SS:USP} and \S \ref{SS:UFJ}, as well as those of the previous section. Combining Beuzart-Plessis-Wan's results on unitary Shalika periods and Corollary \ref{C:SFJ}, we are able to completely determine the multiplicity of unitary FJ periods for discrete series representations. 
\vskip 5pt

Let $V$ be an $N$-dimensional Hermitian space, and $V^+$ the Hermitian space with the same dimension and sign $+1$. Let $\mathscr{W}'=\mathscr{W}_{\psi^E}$ be the Whittaker datum of $\U(V^+)$ determined by the $N_{E/F}(E^\times)$-orbit of the additive character
\[
	\psi^E:=\psi\left(\frac{1}{2}Tr_{E/F}(\delta\cdot~)\right)
\]
of $E/F$. We shall use the local Langlands correspondence for $\U(V)$ relative to $\mathscr{W}'$ to parametrize irreducible discrete series representations of $\U(V)$. Let $V_0$ be an $m$-dimensional Hermitian subspace of $V$. Without loss of generality, we assume that $2m\leq N$, and put $\ell=N-2m$. 
\vskip 5pt

\subsection{Even dimensional case}
We first consider the case that $\ell=0$. In particular, $\dim V=2m$ is even in this case. 
\begin{Thm}\label{T:UFJ-1}
Let $\phi$ be a discrete L-parameter of $\U(V)$, and $(\mu_1,\mu_2)$ a pair of characters of $E_1$. If there exists some $\pi\in\Pi_\phi(\U(V))$ such that $Lin_{V_0}(\pi, \mu_1\boxtimes\mu_2)\neq 0$, then the following holds:
\begin{enumerate}
  \item We have a dichotomy
    \[
      \epsilon(V)=\epsilon\left(\frac{1}{2},\phi\cdot\widetilde{\mu}_2^{-1},\psi_2^E\right).
    \]
  \vskip 5pt
  
  \item As a representation of $WD_E$, $\phi$ takes value in $\GSp_{2m}(\C)$ with similitude factor $\widetilde{\mu}_1\widetilde{\mu}_2$.
\end{enumerate}
\end{Thm}
\vskip 5pt

\begin{proof}
Let $\pi=\pi(\phi,\eta)\in\Pi_\phi(\U(V))$ be such that $Lin_{V_0}(\pi, \mu_1\boxtimes\mu_2)\neq 0$. Since $\pi$ is a discrete series representation of $\U(V)$, it is also unitary. By taking complex conjugation, we get
\[
	Lin_{V_0}\left(\pi^\vee, \mu_1^{-1} \boxtimes \mu_2^{-1}\right)\neq 0.
\]
Let $W$ be the maximally split $2m$-dimensional skew-Hermitian space over $E$, and set $\chi_W=\widetilde{\mu}_2$. We also fix a character $\chi_V$ of $E^\times$ satisfying the condition described in \S \ref{SS:theta-splitting}. Consider the theta correspondence between $\U(V)$ and $\U(W)$ with respect to the additive character $\psi$ and the splitting characters $(\chi_V,\chi_W)$. Then Corollary \ref{C:SFJ} asserts that
\[
	Sha_{V_0^c}\left(\Theta(\pi),\mu'\right)\neq 0,
\]
where $\mu'=\mu_1\mu_2^{-1}\cdot\chi_V|_{E_1}$ is a character of $E_1$. 
In particular, $\Theta(\pi)\neq 0$. It then follows from \cite[\S 4.4 (P1)]{GI2} that our first assertion holds. Indeed, one also knows that $\Theta(\pi)$ is an irreducible discrete series representation of $\U(W)$ with L-parameter
\[
  \theta_{2m}(\phi)=\phi\otimes\widetilde{\mu}_2^{-1}\chi_V.
\]
By Theorem \ref{C:USP-1}, $\theta_{2m}(\phi)$ takes value in $\GSp_{2m}(\C)$ with similitude factor $\widetilde{\mu}'$. This implies our second assertion.
\end{proof}

\vskip 5pt

\begin{Rmk}
\begin{enumerate}
  \item Indeed, to prove the assertion (1) in this theorem, we only need to assume that $\phi$ is tempered. Moreover, if Theorem \ref{C:USP-1} is available for tempered L-parameters, then we can also prove the assertion (2) for tempered L-parameters.
  \vskip 5pt

  \item By considering the theta correspondence between $\SO_{2n+1} \times \Mp_{2n}$, we can also prove an analog of Theorem \ref{T:UFJ-1}(1)  for the $\left(\SO_{2n+1} , \SO_n \times\SO_{n+1}\right)$-period problem.
\end{enumerate}
\end{Rmk}

\vskip 10pt

Now assume that the two conditions in the theorem hold for the discrete L-parameter $\phi$. As in \S \ref{SS:T-USP}, we decompose $\phi$ into a ``symplectic part'' and a ``non-symplectic part'' as
\[
	\phi=\sum_{i\in I_\phi}\phi_i + \sum_{j\in J_\phi}\left(\phi_j + \phi_j^\vee\widetilde{\mu}\right),
\]
where $\mu=\mu_1\mu_2$. Similarly, we let
\[
	\mathcal{S}_{\phi}^\Delta = \sum_{i\in I_\phi}\Z/2\Z \, a_i + \sum_{j\in J_\phi}\Z/2\Z \, \left(b_j + b_j^*\right) \subset \mathcal{S}_{\phi},
\]
where $a_i$, $b_j$ and $b_j^*$ are the basis elements in $\mathcal{S}_{\phi}$ corresponding to $\phi_i$, $\phi_j$ and $\phi_j^\vee\widetilde{\mu}$ respectively. 
Let $\eta^\flat$  be the character of $\mathcal{S}_\phi$ defined by the formula
\[
	\eta^\flat(a)=\epsilon\left(\frac{1}{2},\phi^a \cdot \widetilde{\mu}_2^{-1}, \psi_2^E\right)
\]
for $a\in\mathcal{S}_\phi$.

\begin{Thm} \label{T:UFJ-2}
Suppose that the L-parameter $\phi$ is discrete and satisfies all conditions in Theorem \ref{T:UFJ-1}. Let $\pi=\pi(\phi,\eta)$ be an irreducible discrete series representation of $\U(V)$ belonging to the L-packet of $\phi$.
\begin{enumerate}
\item Suppose that $I_\phi\neq\varnothing$. Then $\pi$ has non-zero $(\mu_1,\mu_2)$-linear periods if and only if
\[
	\eta\cdot\eta^\flat~\Big|_{\mathcal{S}_\phi^\Delta}=1.
\] 
When this condition holds, we have
\[
	\dim Lin_{V_0}(\pi, \mu_1\boxtimes\mu_2)=2^{|I_\phi|-1}.
\]
In particular, the dimension of $Lin_{V_0}(\pi, \mu_1\boxtimes\mu_2)$ is independent of the choice of $V_0$ in this case.
\vskip 5pt

\item Suppose that $I_\phi=\varnothing$. Then $\pi$ has non-zero $(\mu_1,\mu_2)$-linear periods if and only if
\[
	\eta\cdot\eta^\flat~\Big|_{\mathcal{S}_\phi^\Delta}=1, \quad \text{and} \quad \eta\cdot\eta^\flat\left(\sum_{j\in J_\phi}b_j\right)=\epsilon(V_0).
\]
Here $\epsilon(V_0)$ is the sign of the Hermitian space $V_0$. When these conditions hold, we have 
\[
	\dim Lin_{V_0}(\pi, \mu_1\boxtimes\mu_2)=1.
\]
\end{enumerate}
\end{Thm}
\vskip 5pt

\begin{proof}
As in the proof of Theorem \ref{T:UFJ-1}, let $W$ be the $2m$-dimensional maximally split skew-Hermitian space over $E$, we consider the theta correspondence between $\U(V)$ and $\U(W)$ with respect to the additive character $\psi$ and characters $(\chi_V,\chi_W)$, where $\chi_W=\widetilde{\mu}_2$. With $\pi = \pi(\phi,\eta)$, we know by \cite[\S 4.4 (P1)]{GI2} that 
\[  \Theta(\pi) = \pi\left(\theta_{2m}(\phi), \eta \cdot \eta^{\flat}\right), \]
where $\theta_{2m}(\phi)=\phi\otimes\widetilde{\mu}_2^{-1}\chi_V$. Let $\mu'=\mu_1\mu_2^{-1}\cdot\chi_V|_{E_1}$ be a character of $E_1$. By Theorem \ref{C:USP-2} and the hypotheses of the theorem, we know
\[
	\dim Sha_{V_0^c}\left(\Theta(\pi),\mu'\right)
\]
precisely. The assertions of the theorem then follow from  Corollary \ref{C:SFJ}. 
\end{proof}

\vskip 5pt

\subsection{Odd dimensional case}
Next we consider another interesting case, namely the case where $\ell= N -2m = 1$. In particular, $\dim V=2m+1$ is odd in this case. The two theorems in this subsection are parallel with those in the previous subsection.
\begin{Thm}\label{T:UFJ-I}
Let $\phi$ be a discrete L-parameter of $\U(V)$, and $(\mu_1,\mu_2)$ a pair of characters of $E_1$. If there exists some $\pi\in\Pi_\phi(\U(V))$ such that $Lin_{V_0}(\pi, \mu_1\boxtimes\mu_2)\neq 0$ (with $V_0 \subset V$ of dimension $m$), then the following holds:
\begin{enumerate}
  \item The L-parameter $\phi$ contains $\widetilde{\mu}_2$ as a subrepresentation. Equivalently, the local L-function $L\left(s,\phi\cdot\widetilde{\mu}_2^{-1}\right)$ has a pole at $s=0$. 
  \vskip 5pt
  
  \item As a representation of $WD_E$, $\phi-\widetilde{\mu}_2$ takes value in $\GSp_{2m}(\C)$ with similitude factor $\widetilde{\mu}_1\widetilde{\mu}_2$.
\end{enumerate}
\end{Thm}
\vskip 5pt

\begin{proof}
Similar to the proof of Theorem \ref{T:UFJ-1}.
\end{proof}

\vskip 5pt

\begin{Rmk}
\begin{enumerate}
  \item Again, to prove the assertion (1) in this theorem, we only need to assume that $\phi$ is tempered. Moreover, if Theorem \ref{C:USP-1} is available for tempered L-parameters, then we can also prove the assertion (2) for tempered L-parameters.
  \vskip 5pt

  \item This theorem gives an analog of \cite[Thm. 1.4]{PWZ} for unitary groups. Indeed, by considering the theta correspondence between $\SO_{2n} \times \Sp_{2n-2}$, our method is applicable to the $\left(\SO_{2n} , \SO_{n-1} \times \SO_{n+1}\right)$-period problem as well.
\end{enumerate}
\end{Rmk}

\vskip 5pt

Assume that the two conditions in the theorem hold for the discrete L-parameter $\phi$. Again as in \S \ref{SS:T-USP}, we decompose $\phi-\widetilde{\mu}_2$ into a ``symplectic part'' and a ``non-symplectic part'' as
\[
  \phi-\widetilde{\mu}_2=\sum_{i\in I_\phi}\phi_i + \sum_{j\in J_\phi}\left(\phi_j + \phi_j^\vee\widetilde{\mu}\right),
\]
where $\mu=\mu_1\mu_2$. We also put 
\[
  \mathcal{S}_{\phi}^\Delta = \sum_{i\in I_\phi}\Z/2\Z \, a_i + \sum_{j\in J_\phi}\Z/2\Z \, \left(b_j + b_j^*\right) \subset \mathcal{S}_{\phi},
\]
where $a_i$, $b_j$ and $b_j^*$ are the basis elements in $\mathcal{S}_{\phi}$ corresponding to $\phi_i$, $\phi_j$ and $\phi_j^\vee\widetilde{\mu}$ respectively.  \begin{Thm} \label{T:UFJ-II}
Suppose that the L-parameter $\phi$ is discrete and satisfies all conditions in Theorem \ref{T:UFJ-I}. Let $\pi=\pi(\phi,\eta)$ be an irreducible discrete series representation of $\U(V)$ belonging to the L-packet of $\phi$.
\begin{enumerate}
\item Suppose that $I_\phi\neq\varnothing$. Then $\pi$ has non-zero $(\mu_1,\mu_2)$-linear periods if and only if
\[
  \eta~\Big|_{\mathcal{S}_\phi^\Delta}=1, \quad \text{and} \quad \eta(e)=\epsilon(V).
\] 
Here,  $e$ is the basis element in $\mathcal{S}_{\phi}$ corresponding to $\widetilde{\mu}_2$
and $\epsilon(V)$ is the sign of the Hermitian space $V$. When these conditions hold, we have
\[
  \dim Lin_{V_0}(\pi, \mu_1\boxtimes\mu_2)=2^{|I_\phi|-1}.
\]
In particular, the dimension of $Lin_{V_0}(\pi, \mu_1\boxtimes\mu_2)$ is independent of the choice of $V_0$ in this case.
\vskip 5pt

\item Suppose that $I_\phi=\varnothing$. Then $\pi$ has non-zero $(\mu_1,\mu_2)$-linear periods if and only if
\[
  \eta~\Big|_{\mathcal{S}_\phi^\Delta}=1, \quad \eta\left(\sum_{j\in J_\phi}b_j\right)=\epsilon(V_0), \quad \text{and} \quad \eta(e)=\epsilon(V).
\]
Here $\epsilon(V_0)$ is the sign of the Hermitian space $V_0$. When these conditions hold, we have 
\[
  \dim Lin_{V_0}(\pi, \mu_1\boxtimes\mu_2)=1.
\]
\end{enumerate}
\end{Thm}
\vskip 5pt

\begin{proof}
Similar to the proof of Theorem \ref{T:UFJ-2}.
\end{proof}
\vskip 5pt

We end  this section by remarking that we can also establish analogs of these theorems in the case that $\ell>1$ by using results from \cite{AG}.
More precisely, if an irreducible tempered representation $\pi = \pi(\phi,\eta)$ of $\U(V)$ satisfies $Lin_{V_0}(\pi, \mu_1\boxtimes\mu_2)\neq 0$ with $m = \dim V_0 <  (\dim V -1)/2$, then 
$\pi$ has nonzero theta lift to the quasi-split $\U(W)$ with $\dim W = 2m <\dim V -1$. The results of \cite{AG} then implies strong constraints on   $(\phi, \eta)$. Since 
 the precise statements are a bit complicated to formulate, we shall omit them here.

\vskip 10pt

\section{\bf  $\GL_n(E)$-period of $\U_{2n}(F)$} \label{S:variant}
In this section, we consider a variant of the unitary Friedberg-Jacquet period of $\U(W)$. If $W = X + Y$ is a Witt decomposition with $X$ and $Y$ maximal isotropic of dimension $n$, (so $\dim W = 2n$), then 
we have the Levi subgroup $M_{X,Y} \cong \GL(X) \subset P(X)$. For $\pi \in {\rm Irr}(\U(W))$, consider
\[  Lin_X(\pi, \chi) := \Hom_{\GL(X)} (\pi, \chi \circ {\det}_X ) \cong \Hom_{\U(W)}( {\rm ind}_{\GL(X)}^{\U(W)}  \chi^{-1} \circ {\det}_X,  \pi^{\vee}).  \]
This may legitimately be called the unitary linear period.
\vskip 5pt

We shall analyze the induced representation 
\[  \Omega_{\chi}:=  {\rm ind}_{\GL(X)}^{\U(W)}  \chi \circ {\det}_X.  \] 
In particular, we shall show:
\vskip 5pt

\begin{Prop} \label{P:stages}
The representation $\Omega_\chi$ has a $\U(W)$-equivariant filtration
\[  \mathcal{F}_{-1} = \{ 0 \} \subset \mathcal{F}_0 \subset....\subset \mathcal{F}_n= \Omega_{\chi} , \]
whose successive quotients are
\[   \mathcal{F}_k/ \mathcal{F}_{k-1}  \cong I^{\U(W)}_{P(X_k)}  (\chi \cdot |-|^{k/2}\circ {\det}_{X_k} ) \, \boxtimes   \,\left( \bigoplus_{B_k} Sha_{\U(X_k^{\perp}/X_k)}(B_k, \chi|_{E_1})  \right) , \]
where 
\begin{itemize}
\item $X_k\subset  X$ is an isotropic subspace of dimension $k$ and $P(X_k)$ is the maximal parabolic subgroup stabilizing $X_k$, with Levi factor $\GL(X_k) \times \U(X_k^{\perp}/X_k)$;

\item the nondegenerate Hermitian space $X_k^{\perp}/X_k$ has dimension $2n-2k$ and has Witt decomposition 
\[  X_k^{\perp}/X_k = X/X_k \oplus (Y \cap X_k^{\perp})/X_k; \]
\item  $B_k$ runs over the (two if $k < n$)  equivalence classes of Hermitian forms on $X$ with kernel $X_k$, so that $B_k$ induces a nondegenerate Hermitian form on $X /X_k$ and thus may also be considered  as a nondegenerate Hermitian space of dimension $n-k$ or an element in the unipotent radical of the Siegel parabolic of $\U(X_k^{\perp}/X_k)$ stabilziing the maximal isotropic subspace $(Y \cap X_k^{\perp})/X_k$;

\item $I_{P(X_k)}^{\U(W)}$ stands for normalized parabolic induction;

\item  $Sha_{\U(X_k^{\perp}/X_k)}\left(B_k, \chi|_{E_1}\right)$ is the Shalika module of $\U(X_k^{\perp}/X_k)$ relative to $B_k$:
\[ Sha_{\U(X_k^{\perp}/X_k)}(B_k , \chi|_{E_1})  =  {\rm ind}^{\U(X_k^{\perp}/X_k)}_{S(B_k)}  \chi \circ {\det}_{X/X_k} \boxtimes \psi_{B_k} , \]
where $S_{B_k}$ is the Shalika subgroup of $\U(X_k^{\perp}/X_k)$  associated to $B_k$.
\end{itemize}
\vskip 5pt

\noindent  In particular, the bottom piece of the filtration is the submodule
\[  \mathcal{F}_0 =  \bigoplus_B  Sha_{\U(W)}(B, \chi|_{E_1})  . \]
\end{Prop}
\vskip 5pt

\begin{proof}
By induction in stages, 
\begin{equation}\label{E:stages}
 \Omega_{\chi} =  {\rm ind}_{\GL(X)}^{\U(W)}  \chi \circ {\det}_X \cong {\rm ind}_{P(X)}^{\U(W)}\left(   (\chi \circ {\det}_X) \cdot (  {\rm ind}_{\GL(X)}^{P(X)}  1) \right). \end{equation}
The induced representation ${\rm ind}_{\GL(X)}^{P(X)} 1$ is realized on the space $C^{\infty}_c(N(X))$ of Schwarz functions on $N(X)$, with the following action of $P(X)$:
\[
\begin{cases} 
  (m \cdot f) (n(x)) = f ( m^{-1} n(x) m), &\text{  if $m \in \GL(X)$;} \\
  (n(a) \cdot f)(n(x)) = f ( n(a+x)), &\text{  if $n(a) \in N(X)$.} \end{cases} \]
We change this model by using the Fourier transform:
\[  \mathcal{F}:  C^{\infty}_c(N(X)) \longrightarrow C^{\infty}_c(N(Y))  \]
defined by
\[  \mathcal{F}(f)( n(b))  = \int_{N(X)}  f( n(x)) \cdot \overline{\psi\left(\frac{1}{2}{\rm Tr}_X( y \circ x)\right)} \, dx.   \] 
Transporting the action  over by $\mathcal{F}$, one obtains the following action of $P(X)$ on $C^{\infty}_c(N(Y))$:
\[  \begin{cases}
(m \cdot f)(n(y)) =  |{\det}_X(m)|^n  \cdot f ( m^{-1} \cdot n(y))  , &\text{  if $m \in \GL(X)$;} \\
(n(a) \cdot f)( n(y)) = \psi_y(a)  \cdot f(n(y)), &\text{  if $n(a) \in N(X)$.} \end{cases} \]  
\vskip 5pt

Now we consider the orbits of the $\GL(X)$-action on $N(Y)$, which provide a stratification of $N(Y)$.  Recall that elements of $N(Y)$ can be regarded as Hermitian forms on $X$. Thus, there are two open orbits, given by the two equivalence classes of nondegenerate Hermitian forms, whereas the degenerate orbits correspond to Hermitian forms which are degenerate. The following lemma summarizes the result:
\vskip 5pt

\begin{Lem}
For each integer $ 0 \leq k \leq n = \dim X$, there are two (if $k < n$) or one (if $k =n$)  $\GL(X)$-orbits of Hermitian forms on $X$ whose kernel $X_k$ has dimension $k$, determined by the nondegenerate Hermitian forms  induced on $X/X_k$. For such a Hermitian form $B_k$ on $X$ (which induces a nondegenerate Hermitian form on $X/X_k$, still denoted by $B_k$), the corresponding stabilizer in $\GL(X)$ is the subgroup of the parabolic subgroup
\[  \mathcal{Q}(X_k) = (\GL(X_k) \times \GL(X/X_k)) \ltimes \mathcal{U}(X_k) \subset \GL(X)  \]  
stabilizing $X_k$,  given by
\[  \mathcal{Q}(X_k,B_k) := \left(  \GL(X_k) \times \U(X/X_k, B_k) \right) \ltimes \mathcal{U}(X_k). \]
\end{Lem}
\vskip 5pt

As a consequence of the lemma, one has:
\vskip 5pt

\begin{Lem}
The stratification of $N(Y)$ by $\GL(X)$-orbits gives rise to a $P(X)$-equivariant filtration  
\[  \mathcal{S}_{-1} =  \{0 \} \subset \mathcal{S}_0  \subset......\subset \mathcal{S}_n =  {\rm ind}_{\GL(X)}^{P(X)} 1 , \]
whose successive quotients  are given as follows. For $0 \leq k \leq n$,
\[    \mathcal{S}_k / \mathcal{S}_{k-1} \cong \bigoplus_{B_k} {\rm ind}^{P(X)}_{\mathcal{Q}(X_k,B_k) \ltimes N(X)}   |{\det}_{X_k}|^n \boxtimes \psi_{B_k}   ,  \]
where the sum over $B_k$ runs over the two equivalence classes of nondegenerate Hermitian spaces on $X/X_k$ and the subgroup $\U(X/X_k, B_k) \ltimes \mathcal{U}(X_k) \subset \mathcal{Q}(X_k, B_k)$ acts trivially on the inducing module.
\end{Lem}
\vskip 10pt

When we use the result of the above lemma in (\ref{E:stages}), we see that $\Omega_{\chi}$ has a $\U(W)$-equivariant filtration
\[  \mathcal{F}_{-1} = \{ 0 \} \subset \mathcal{F}_0 \subset....\subset \mathcal{F}_n= \Omega_\chi , \]
whose successive quotients are
\[   \mathcal{F}_k/ \mathcal{F}_{k-1}  \cong {\rm ind}^{\U(W)}_{\mathcal{Q}(X_k, B_k) \ltimes N(X)} (\chi \circ {\det}_{X}) |{\det}_{X_k}|^n \boxtimes \psi_{B_k}.   \]
Observe that 
\[  \mathcal{Q}(X_k, B_k) \ltimes N(X) \subset P(X_k)  ,  \]
where $P(X_k)$ is the maximal parabolic subgroup of $\U(W)$ stabilizing $X_k$, with Levi factor $\GL(X_k) \times \U(X_k^{\perp}/X_k)$, 
Moreover, we can re-express $\mathcal{Q}(X_k, B_k)\ltimes N(X)$ as:
\[  \mathcal{Q}(X_k, B_k)\ltimes N(X) = (\GL(X_k)  \times  S_{B_k} ) \ltimes N(X_k) , \]
where $S_{B_k}$ is the Shalika subgroup of $\U(X_k^{\perp}/X_k)$ associated to $B_k$. 

\vskip 10pt
Hence, by induction in stages, we have:
\[   \mathcal{F}_k/ \mathcal{F}_{k-1}  \cong  \bigoplus_{B_k} {\rm Ind}^{\U(W)}_{P(X_k)}  (\chi \cdot |-|^n \circ {\det}_{X_k} ) \, \boxtimes   \,Sha_{\U(X_k^{\perp}/X_k)}\left(B_k, \chi|_{E_1}\right) , \]
where we are using  unnormalied parabolic induction here, and
\[ 
	Sha_{\U(X_k^{\perp}/X_k)}(B_k, \chi|_{E_1}) = {\rm ind}^{\U(X_k^{\perp}/X_k)}_{S_{B_k}}  \chi \circ {\det}_{X/X_k} \boxtimes \psi_{B_k}  
\]
is the Shalika module of $\U(X_k^{\perp}/X_k)$ relative to $B_k$. Taking note that the modulus character of $P(X_k)$ is given by $|{\det}_{X_k}|^{2n-k}$ when restricted to $\GL(X_k)$, we see that
\[   \mathcal{F}_k/ \mathcal{F}_{k-1}  \cong \bigoplus_{B_k}  I^{\U(W)}_{P(X_k)}  (\chi \cdot |-|^{k/2}\circ {\det}_{X_k} ) \, \boxtimes   \,Sha_{\U(X_k^{\perp}/X_k)}(B_k, \chi|_{E_1}),   \]
where we are using normalized parabolic induction here.
This completes the proof of Proposition \ref{P:stages}.
\end{proof}
\vskip 10pt

We now have:
 
 \vskip 5pt
\begin{Thm}  \label{T:variant}
Let $W = X \oplus Y$ be a $2n$-dimensional skew-Hermitian space with the given  Witt decomposition, and let
$\pi \in {\rm Irr}(\U(W))$ be tempered. Then for any unitary character $\chi$ of $E^{\times}$, 
\[  \Hom_{\GL(X)}(\pi, \chi\circ {\det}_X)  \cong \bigoplus_B \Hom_{S_B}( \pi, \chi \circ {\det}_{X} \boxtimes \psi_B) , \]
where the sum runs over the two equivalence classes of  nondegenerate Hermitian spaces of dimension $n$. 
\end{Thm}
\vskip 5pt

\begin{proof}
For $\pi \in {\rm Irr}(\U(W))$, 
\[  \Hom_{\GL(X)}(\pi, \chi\circ {\det}_X) \cong  \Hom_{\U(W)} ( \Omega_{\chi^{-1}}, \pi^{\vee}). \]
By Proposition \ref{P:stages}, we need to consider
\[ 
	\Hom_{\U(W)} \left(   I^{\U(W)}_{P(X_k)}  (\chi^{-1} \cdot |-|^{k/2}\circ {\det}_{X_k} ) \, \boxtimes   \,Sha_{\U(X_k^{\perp}/X_k)}(B_k, \chi^{-1}|_{E_1})  , \pi^{\vee} \right), 
\]
which, by the second adjointness theorem of Bernstein, is equal to
\[  \Hom_{\GL(X_k) \times \U(X_k^{\perp})} \left( \chi^{-1} \cdot |-|^{k/2}\circ {\det}_{X_k}  \, \boxtimes   \,Sha_{\U(X_k^{\perp}/X_k)}(B_k, \chi^{-1}|_{E_1}), R_{\overline{P(X_k)}}(\pi^{\vee}) \right). \]
Here, $R_{\overline{P(X_k)}}(\pi^{\vee})$ is the normalized Jacquet module of $\pi^{\vee}$ with respect to the opposite parabolic $\overline{P(X_k)}$.
\vskip 5pt

Now since $\pi^{\vee}$ is tempered, the central exponents of $R_{\overline{P(X_k)}}(\pi^{\vee})$, as unramified characters of $\GL(X_k)$, are of the form $|{\det}_{X_k}|^t$ with $t \leq 0$. On the other hand, the central exponent of $\chi^{-1} \cdot |-|^{k/2}\circ {\det}_{X_k}  \, \boxtimes   \,Sha_{\U(X_k^{\perp}/X_k)}(B_k, \chi^{-1}|_{E_1})$ is equal to $|{\det}_{X_k}|^{k/2}$. Thus, if $k>0$, the Hom space above vanishes and likewise
\[   {\rm Ext}^1_{\GL(X_k) \times \U(X_k^{\perp})} \left( \chi^{-1} \cdot |-|^{k/2}\circ {\det}_{X_k}  \, \boxtimes   \,Sha_{\U(X_k^{\perp}/X_k)}(B_k,\chi^{-1}|_{E_1}), R_{\overline{P(X_k)}}(\pi^{\vee}) \right) = 0. \]
In other words, only the term $k =0$ in Proposition \ref{P:stages} contributes and the theorem follows. 
\end{proof}
\vskip 5pt

In view of the results of Beuzart-Plessis and Wan recalled in Theorem \ref{C:USP-1} and Theorem \ref{C:USP-2}, the above theorem allows one to  determine $\dim  \Hom_{\GL(X)}(\pi, \chi\circ {\det}_X)$ precisely for a discrete series representation $\pi$. 
\vskip 5pt

\begin{Thm}\label{T:Variant-ULP}
Let $\pi=\pi(\phi,\eta)$ be an irreducible discrete series representation of $\U(W)$, and $\chi$ a character of $E_1$. Then $Lin_X(\pi, \chi)\neq 0$ if and only if the following two conditions hold:
\begin{enumerate}
\item As a representation of $WD_E$, $\phi$ takes value in $\GSp_{2n}(\C)$ with similitude factor $\widetilde{\chi}$.
\vskip 5pt

\item Let $\mathcal{S}_\phi^\Delta$ be the subgroup defined as in \S \ref{SS:T-USP}. The character $\eta$ satisfies
\[
	\eta~\Big|_{\mathcal{S}_\phi^\Delta}=1.
\] 
\end{enumerate}
When these two conditions hold, we have
\[
	\dim Lin_X(\pi, \chi)=2^{|I_\phi|},
\]
where $I_{\phi}$ is as defined in  \S \ref{SS:T-USP}, so that $|I_{\phi}|$ is the number of ``symplectic summands" in $\phi$.
\end{Thm}
\vskip 5pt

We also remark that the argument in this section applies equally well in the case of $\GL_{2n}$, giving a comparison between the linear period and the Shalika period (at least for tempered representations) which is perhaps more direct than the approach of theta correspondence given in 
\cite[\S 3 and \S4]{G} (which is the $\GL$-analog of Corollary \ref{C:SFJ}).
\vskip 15pt

\section{\bf Global Periods} \label{S:global}
In this section, we consider the global setting. Let $F$ be a number field and $E$ a quadratic field extension of $F$. We denote by $\A$ and $\A_E$ the adele ring of $F$ and $E$ respectively. As in \S \ref{S:prelim}, we shall use $c$ to denote the non-trivial element of $\Gal(E/F)$, and $\omega_{E/F}$ to denote the quadratic character of $F^\times\backslash\A$ associated to $E/F$ by global class field theory. We fix a non-trivial additive character $\psi$ of $F\backslash\A$. For an algebraic group $G$ over $F$, we use the symbol $[G]$ to denote $G(F)\backslash G(\A)$. 
\vskip 5pt

\subsection{\bf Global unitary Shalika period integral}
We first introduce the global unitary Shalika period. Let $W$ be a skew-Hermitian space over $E$. Assume that
\[
	W=X\oplus Y
\]
is a Witt decoposition of $W$. As in \S \ref{SS:USP}, we consider the Siegel parabolic subgroup $P(X)=\GL(X)\cdot N(X)$ stabilizing $X$. For a given (non-degenerate) Hermitian form $B$ on $X$, one has an associated automorphic character of $N(X)$
\[
	\psi_B: N(X)(F)\backslash N(X)(\A)\longrightarrow\C^\times,\quad n(A)\longmapsto\psi\left(\frac{1}{2}Tr_X(A\circ B)\right)
\]
for $n(A)\in N(X)(\A)$. If $f$ is an automorphic form of $\U(W)$, the $\left(N(X),\psi_B\right)$-coefficient of $f$ is a function on $\U(W)(\A)$ defined by 
\[
  \left(\mathcal{F}_{\psi_B}f\right)(g):=\int_{[N(X)]} f(ng) \cdot \overline{\psi_B(n)}\,dn
\]
for $g\in\U(W)(\A)$. 
\vskip 5pt

Consider the stabilizer of $\psi_B$ in $\GL(X)(\A)$, which is the isometry group $\U(X,B)(\A)$. We shall call the (algebraic) subgroup
\[
	S_B := \U(X,B) \ltimes N(X) \subset P(X)
\]
the Shalika subgroup relative to $B$. For any (unitary) automorphic character $\mu$ of $E_1$, $(\mu\circ \det_X) \boxtimes \psi_B$ defines an automorphic character of $S_B$. If $f$ is a cusp form of $\U(W)$, the $(B, \mu)$-Shalika period integral of $f$ is defined by:
\begin{align*}
	Sha_B(f,\mu):=&\int_{[\U(B)\ltimes N(X)]}f(ng)\cdot\overline{\mu\left({\det}_X(g)\right)}\cdot\overline{\psi_B(n)}\,dn\,dg\\
  =&\int_{[\U(B)]}\left(\mathcal{F}_{\psi_B}f\right)(g)\cdot\overline{\mu\left({\det}_X(g)\right)}\,dg.
\end{align*}
The convergence of this integral is guaranteed by the cuspidality of $f$. We say a cuspidal representation $\pi$ of $\U(W)$ has non-zero Shalika period if $Sha_B(f,\mu)\neq 0$ for some $f\in\pi$.
\vskip 5pt

\subsection{\bf Global unitary Friedberg-Jacquet period integral}
Next we introduce the global unitary Friedberg-Jacquet period. Let $V$ be a Hermitian space over $E$. Assume that
\[
	V=V_0 \oplus V_0^\perp,
\]
where $V_0$ is a nondegenerate Hermitian subspace of $V$. As in \S \ref{SS:UFJ}, we set 
\[
	H_{V_0} := \U(V_0) \times \U(V_0^{\perp}) \subset \U(V).
\]
Fix two (unitary) automorphic characters $\mu_1$ and $\mu_2$ of $E_1$. If $f$ is a cusp form of $\U(V)$, the $(\mu_1, \mu_2)$-linear period integral of $f$ is defined by:
\[
	Lin_{V_0}(f, \mu_1\boxtimes\mu_2):=\int_{[\U(V_0)\times \U\left(V_0^\perp\right)]}f(gh)\cdot\overline{\mu_1({\det}_{V_0}(g))}\cdot\overline{\mu_2({\det}_{V_0^{\perp}}(h))}\,dg\,dh .
\]
The convergence of this integral is again guaranteed by the cuspidality of $f$. We say a cuspidal representation $\pi$ of $\U(V)$ has non-zero $(\mu_1, \mu_2)$-linear period (or unitary FJ period) if $Lin_{V_0}(f, \mu_1\boxtimes\mu_2)\neq 0$ for some $f\in\pi$.
\vskip 5pt

By Fubini's theorem, we can decompose $Lin_{V_0}(f, \mu_1\boxtimes\mu_2)$ into a double integral
\[
  Lin_{V_0}(f, \mu_1\boxtimes\mu_2)=\int_{[\U\left(V_0^\perp\right)]}\mathcal{P}_{V_0}\left(h\cdot f,\mu_1\right)\cdot\overline{\mu_2({\det}_{V_0^{\perp}}(h))}\,dh ,
\]
where
\[
  \mathcal{P}_{V_0}\left(f,\mu_1\right)=\int_{[\U(V_0)]}f(g)\cdot\overline{\mu_1({\det}_{V_0}(g))}\,dg
\]
is the $\left(\U(V_0),\mu_1\circ{\det}_{V_0}\right)$-period integral of $f$. Likewise, we have
\[
  Lin_{V_0}(f, \mu_1\boxtimes\mu_2)=\int_{[\U\left(V_0\right)]}\mathcal{P}_{V_0^\perp}\left(g\cdot f,\mu_2\right)\cdot\overline{\mu_1({\det}_{V_0}(g))}\,dg .
\]
\vskip 5pt

\subsection{\bf Transfer of global periods}
This subsection is aimed at establishing a global analog of Theorem \ref{T:SFJ}. Let $W = X +Y$ be a maximally split skew-Hermitian space and $V$ a Hermitian space over $E$. We fix a pair of automorphic characters $(\chi_V, \chi_W)$ of $\A_E^{\times}$ such that:
\[ 
	\chi_V|_{\A_F^{\times}} = \omega_{E/F}^{\dim V} \quad \text{and} \quad \chi_W|_{\A_F^{\times}} = \omega_{E/F}^{\dim W}.
\]
With respect to the additive character $\psi$ and the pair of characters $(\chi_V, \chi_W)$, we may consider the global theta correspondence between $\U(V)$ and $\U(W)$. We will work with the Schr\"odinger model of the Weil representation $\Omega$ relative to the Siegel parabolic $P(X)$, which is realized on the space $\mathcal{S}\left((Y\otimes V)(\A)\right)$ of Schwarz functions on $Y\otimes V = \Hom_E(X^c, V)$, and the action of $\U(V)(\A)\times P(X)(\A)$ is given by the same formulas as (\ref{F:Schro-1}), (\ref{F:Schro-2}) and (\ref{F:Schro-3}). 
\vskip 5pt

Let $\pi$ be a cuspidal representation of $\U(V)$. Consider its global theta lift $\Theta(\pi)$ to the group $\U(W)$, i.e. the subspace of automorphic forms generated by
\[
	 \theta(\varphi,f)(g)=\int_{[\U(V)]}\theta_\varphi(g,h) \cdot \overline{f(h)}\,dh,
\] 
where $\varphi\in\mathcal{S}\left((Y\otimes V)(\A)\right)$, $f\in\pi$, and
\[
	\theta_\varphi(g,h)=\sum_{T\in (Y\otimes V)(F)}\left(\omega(g,h)\cdot\varphi\right)(T) .
\] 
\vskip 5pt

Firstly we compute the $\left(N(X),\psi_B\right)$-coefficient of $\theta(\varphi,f)$. By definition, we have
\[
  \mathcal{F}_{\psi_B}\left(\theta(\varphi,f)\right) (g)=\int_{[N(X)]}\left(\int_{[\U(V)]}\theta_\varphi(ng,h)\overline{f(h)}\,dh\right)\overline{\psi_B(n)}\,dn.
\]
By exchanging the order of integration (we can do this since $[N(X)]$ is compact), we get
\[
  \mathcal{F}_{\psi_B}\left(\theta(\varphi,f)\right) (g)=\int_{[\U(V)]}\left(\int_{[N(X)]}\theta_\varphi(ng,h)\overline{\psi_B(n)}\,dn\right)\overline{f(h)}\,dh,
\]
where the inner integration can be further simplified as 
\begin{align*}
  \int_{[N(X)]}\theta_\varphi(ng,h)\overline{\psi_B(n)}\,dn&=\int_{[N(X)]}\sum_{T\in (Y\otimes V)(F)}\left(\omega(ng,h)\cdot\varphi\right)(T)\cdot\overline{\psi_B(n)}\,dn\\
  &=\sum_{T\in (Y\otimes V)(F)}\int_{[N(X)]}\left(\omega(g,h)\cdot\varphi\right)(T)\cdot\psi_{T^*(V)}(n)\cdot\overline{\psi_B(n)}\,dn\\
  &=\Vol([N(X)])\cdot\sum_{T\in \mathcal{O}_B(F)}\left(\omega(g,h)\cdot\varphi\right)(T).
\end{align*}
Here $\mathcal{O}_B \subset \Hom_E(X^c,V)$ is the closed $F$-subvariety defined by the equation
\[  
  \mathcal{O}_B = \{T : T^*(V) = B^c \},
\]
and $\Vol([N(X)])$ is the volume of $[N(X)]$. For simplicity, we assume that $\Vol([N(X)])=1$. 
\vskip 5pt

Hence, if $\mathcal{O}_B$ is empty, i.e. if there is no embedding $B^c \hookrightarrow V$, then 
\[
  \mathcal{F}_{\psi_B}\left(\theta(\varphi,f)\right)=0. 
\]
On the other hand, if $\mathcal{O}_B$ is nonempty, then Witt's theorem says that $\U(V)(F)$ acts transitively on $\mathcal{O}_B(F)$. If we fix a base point $j \in \mathcal{O}_B(F)$, and write $V = j(X^c) \oplus j(X^c)^{\perp}$, then the orbit map of $j$ induces a bijection
\[
  \mathcal{O}_B(F)\cong \U(V)(F)\big/ \U\left(j(B^c)^\perp\right)(F).
\]
Substituting these equalities back, the integration can be evaluated as follows:
\begin{align*}
  \mathcal{F}_{\psi_B}&\left(\theta(\varphi,f)\right) (g)\\
  &=\int_{[\U(V)]}\sum_{\gamma\in \U(j(B^c)^\perp)(F)\backslash\U(V)(F)}\left(\omega(g,h)\cdot\varphi\right)(\gamma^{-1} j)\cdot\overline{f(h)}\,dh\\
  &=\int_{\U(j(B^c)^\perp)(\A)\backslash\U(V)(\A)}\int_{[\U(j(B^c)^\perp)]}\left(\omega\left(h_1\right)\cdot\omega(g,h)\varphi\right)(j)\cdot\overline{f(h_1h)}\,dh_1\,dh\\
  &=\int_{\U(j(B^c)^\perp)(\A)\backslash\U(V)(\A)}(\omega(g,h)\varphi)(j)\\
  &\qquad\qquad\times\Bigg(\int_{[\U(j(B^c)^\perp)]}\left(h\cdot\overline{f}\right)(h_1)\cdot\chi_W\circ i^{-1}\left({\det}_{j(X^c)^\perp}(h_1)\right)\,dh_1\Bigg) \,dh.
\end{align*}
Here in the last equality, we have made use of the fact that $h_1\in \U\left(j(B^c)^\perp\right)(\A)$ stabilizes $j$ and the formula (\ref{F:Schro-1}) 
\[
  \omega\left(h_1\right)\cdot\Big(\omega(g,h)\varphi\Big)(j)=\chi_W\circ i^{-1}\left({\det}_{j(X^c)^\perp}(h_1)\right)\cdot\omega(g,h)\varphi(j).
\] 
Notice that the inner integration of the last equality is nothing but the $\mathcal{P}_{j(X^c)^\perp}\left(h\cdot\overline{f},\chi_W^{-1}\circ i^{-1}\right)$. We summarize these computations as follows.
\begin{Prop} \label{P:G-SFJ}
Fix a nondegenerate Hermitian form $B$ on $X$. Given an irreducible cuspidal representation $\pi$ of $\U(V)$, we have
\begin{align*}\tag{\dag}\label{Eqn.dagger-0}
\mathcal{F}_{\psi_B}\left(\theta(\varphi,f)\right) (g)=\int_{\U(j(B^c)^\perp)(\A)\backslash\U(V)(\A)}(\omega(g,h)\varphi)(j)\cdot\mathcal{P}_{j(X^c)^\perp}\left(h\cdot\overline{f},\chi_W^{-1}\circ i^{-1}\right)\,dh
\end{align*}
for $\varphi\in\mathcal{S}\left((Y\otimes V)(\A)\right)$ and $f\in\pi$. Consequently, the cuspidal representation $\pi^\vee$ has non-zero $\left(\U\left(j(B^c)^\perp\right), \chi_W^{-1}\circ i^{-1}\circ{\det}_{j(X^c)^\perp} \right)$-period if and only if $\Theta(\pi)$ has non-zero $\left(N(X),\psi_B\right)$-coefficient.
\end{Prop}
\vskip 5pt

\begin{proof}
The ``if'' part is trivial, so we omit it here. For the ``only if'' part, suppose that $\pi^\vee$ has non-zero 
$\left(\U\left(j(B^c)^\perp\right), \chi_W^{-1}\circ i^{-1}\circ{\det}_{j(X^c)^\perp} \right)$-period. Then there exists some $f\in\pi$, such that
\[
  \mathcal{P}_{j(X^c)^\perp}\left(\overline{f},\chi_W^{-1}\circ i^{-1}\right)\neq0.
\]
Then the map
\[
  h\longmapsto \mathcal{P}_{j(X^c)^\perp}\left(h\cdot\overline{f},\chi_W^{-1}\circ i^{-1}\right)
\]
for $h\in\U(V)(\A)$ gives a non-zero function $C^\infty\left(\U\left(j(B^c)^\perp\right)(\A)\backslash \U(V)(\A),\chi_W^{-1}\circ i^{-1}\circ{\det}_{j(X^c)^\perp}\right)$. On the other hand, since $\mathcal{O}_B$ is a closed $F$-subvariety of $Y\otimes V\cong\Hom_E(X^c, V)$, the map
\[
  \varphi\longmapsto\Big(h\mapsto(\omega(h)\varphi)(j)\Big)
\]
gives a surjection between spaces of Schwartz sections
\[
  \mathcal{S}\left((Y\otimes V)(\A)\right)\longrightarrow \mathcal{S}\left(\U\left(j(B^c)^\perp\right)(\A)\backslash \U(V)(\A),\chi_W^{-1}\circ i^{-1}\circ{\det}_{j(X^c)^\perp}\right).
\]
In particular, we can choose an appropriate $\varphi\in\mathcal{S}\left((Y\otimes V)(\A)\right)$ such that
\[
  \mathcal{F}_{\psi_B}\left(\theta(\varphi,f)\right) (1)=\int_{\U(j(B^c)^\perp)(\A)\backslash\U(V)(\A)}(\omega(h)\varphi)(j)\cdot\mathcal{P}_{j(X^c)^\perp}\left(h\cdot\overline{f},\chi_W^{-1}\circ i^{-1}\right)\,dh \neq 0.
\]
This completes the proof.
\end{proof}
\vskip 5pt

As in the local case, there are some extra symmetries here. By keeping track of the $\U(B)(\A)$-actions on various objects, we can deduce a similar formula relating unitary Shalika period integrals and unitary FJ period integrals. 
\vskip 5pt

\begin{Thm} \label{T:G-SFJ}
In the context of Proposition \ref{P:G-SFJ}, if $\Theta(\pi)$ is contained in the space of cusp forms of $\U(W)$, then
\begin{align*}\tag{\dag\dag}\label{Eqn.dagger}
Sh&a_B\left(\theta(\varphi,f),\mu\right)\\
&=\int_{\mathcal{O}_B(\A)}(\omega(g,h)\varphi)(j)\cdot\overline{\mu({\det}_X(g))}\cdot Lin_{j(X^c)}\left(h\cdot\overline{f}\,,\, \widetilde{\mu}^{-1}\cdot\chi_V^{2}\chi_W^{-1}\circ i^{-1}\boxtimes\chi_W^{-1}\circ i^{-1}\right)\,dg\,dh
\end{align*}
for $\varphi\in\mathcal{S}\left((Y\otimes V)(\A)\right)$ and $f\in\pi$. Here the symbol
\[
  \int_{\mathcal{O}_B(\A)}\cdot\,\,dg\,dh
\]
is the integration over $\left(\U(j(B^c))^{\Delta} \times \U\left(j(B^c)^\perp\right)\right)(\A)\big\backslash (\U(B)\times \U(V))(\A)$, and $\U(j(B^c))^{\Delta}$ is the diagonally embedded subgroup 
\[   
  \begin{CD} 
    \Delta: \U(j(B^c)) @>({\rm Ad}(j) \circ \iota)^{-1} \times {\rm id}>> \U(B) \times \U(j(B^c)) @>>> \U(B) \times \U(V)   
  \end{CD} 
\]
similar to that described in \S \ref{SS:TOP}. Consequently, under the assumption the cuspidal representation $\pi^\vee$ has non-zero $\left(\widetilde{\mu}^{-1}\cdot\chi_V^{2}\chi_W^{-1}\circ i^{-1}, \chi_W^{-1}\circ i^{-1}\right)$-linear period if and only if $\Theta(\pi)$ has non-zero $(B,\mu)$-Shalika period.
\end{Thm}
\vskip 5pt

\begin{proof}
Similar to the proof of Proposition \ref{P:G-SFJ}.
\end{proof}
\vskip 5pt

\begin{Rmk}
\begin{enumerate}
  \item In this theorem, the only reason for making a cuspidality assumption on $\Theta(\pi)$ is to guarantee the convergence of the integral $Sha_B\left(\theta(\varphi,f),\mu\right)$. Notice that if $\U(B)$ is anisotropic, then $[\U(B)]$ is compact and the integral $Sha_B\left(\theta(\varphi,f),\mu\right)$ is always convergent. Hence, in this special case, there is no need to make any assumption on $\Theta(\pi)$.
  \vskip 5pt

  \item Indeed, all computations in this subsection can also be carried out for symplectic-orthogonal and metaplectic-orthogonal dual pairs. 
\end{enumerate}
\end{Rmk}

\vskip 5pt

\subsection{\bf Application: a conjecture of Xiao-Zhang}
In a recent work of Jingwei Xiao and Wei Zhang, they studied the global Friedberg-Jacquet period for unitary groups using the relative trace formula. In the spirit of Theorem \ref{T:split-global}(ii), they gave the following conjecture:
\begin{Conj}\label{C:X-Z}
Let $V$ be a $2n$-dimensional Hermitian space and $V_0$ be an $n$-dimensional nondegenerate subspace of $V$. Let $\pi$ be an irreducible tempered cuspidal representation of $\U(V)$. Then $\pi$ has non-zero $(\mu_1,\mu_2)$-linear period (with respect to $V_0$) if and only if the following holds.
\begin{enumerate}
	\item The $L$-function $L\left(s,\pi\times\widetilde{\mu}_2^{-1}\right)$ is non-vanishing at $s=1/2$, where $\widetilde{\mu}_2$ is the pull-back of $\mu_2$ along the natural map
	\[
		\A_E^\times/\A^\times\longrightarrow E_1(\A).
	\]
  \vskip 5pt

  \item The standard base change  $\BC(\pi)$ of $\pi$ to $\GL(V)$ is of symplectic similitude type with similitude factor $\widetilde{\mu}_1 \widetilde{\mu}_2$; more precisely, one has an isobaric  decomposition
  \[  \BC(\pi) = \rho \boxplus  \left( \boxplus_{i\in I}  \Pi_i \right) \boxplus \rho^{\vee} \widetilde{\mu}_1 \widetilde{\mu}_2,  \]
  where for each $i \in I$,  $\Pi_i$ is a cuspidal representation of an appropriate $\GL_{n_i}$ such that  the twisted exterior square L-function $L\left(s, \Pi_i, \bigwedge^2\otimes\left(\widetilde{\mu}_1\widetilde{\mu}_2\right)^{-1}\right)$ has a pole at $s=1$.
    \vskip 5pt

\item For all place $v$ of $F$, $\pi_v$ has non-zero $(\mu_{1,v},\mu_{2,v})$-linear period.
\end{enumerate} 
\end{Conj}
\vskip 5pt

In fact, this conjecture may be viewed as a special case of a principle proposed by Getz-Wambach in \cite{GW}. Their principle asserts that for an irreducible cuspidal representation $\pi$ of $\U(V)$, the nonvanishing of the $(\mu_1,\mu_2)$-linear period is roughly equivalent to the non-vanishing of $(\widetilde{\mu}_1,\widetilde{\mu}_2)$-linear period for $\BC(\pi)$. In \cite[Cor. 1.3(3)]{PWZ}, Pollack-Wan-Zydor have proved the ``only if'' part of this conjecture under certain conditions on $\pi$ (for example, they assumed that $\pi$ is globally generic).
\vskip 5pt

We also remark that the characters $\mu_1$ and $\mu_2$ play symmetrical roles in the setup of Conjecture \ref{C:X-Z}, so it may seem a bit odd that $\mu_2$ is singled out in Conjecture \ref{C:X-Z}(1).
In fact, under Conjecture \ref{C:X-Z}(2), the nonvanishing of  $L\left(1/2,\pi\times\widetilde{\mu}_2^{-1}\right)$  is equivalent to that  of  $L\left(1/2,\pi\times\widetilde{\mu}_1^{-1}\right)$. To see this, note that Conjecture \ref{C:X-Z}(2) implies that
\[  \BC(\pi)   \cong \BC(\pi)^{\vee} \otimes \widetilde{\mu}_1\widetilde{\mu}_2. \]
Now the desired equivalence of nonvanishing follows from this observation by an application of the global functional equation for the L-function $L(s, \pi \times \widetilde{\mu}_2^{-1}) = L(s, \BC(\pi) \times \widetilde{\mu}_2^{-1})$:
\begin{align*}
	L\left(s,\BC(\pi)\times\widetilde{\mu}^{-1}_2\right)&=\epsilon\left(s,\BC(\pi)\times\widetilde{\mu}^{-1}_2\right)\cdot L\left(1-s,\BC(\pi)^\vee\times\widetilde{\mu}_2\right)\\
	&=\epsilon\left(s,\BC(\pi)\times\widetilde{\mu}^{-1}_2\right)\cdot L\left(1-s,\BC(\pi) \times\widetilde{\mu}^{-1}_1\right).
\end{align*}
\vskip 5pt

As an application of Theorem \ref{T:G-SFJ}, here we give a proof of the ``only if'' part of this conjecture using the theta correspondence, under some assumption on a local component of $\pi$ (which also intervenes in \cite{PWZ}; see Corollary \ref{C:XZW} below). We shall work with a slightly more general setting first.  

\begin{Prop} \label{P:XZ-1}
Let $V$ be a $2n$-dimensional Hermitian space and $V_0$ be an $n$-dimensional nondegenerate subspace of $V$. Let $\pi$ be an irreducible cuspidal representation of $\U(V)$ such that $\pi_v$ is generic at some finite place $v$ of $F$. If $\pi$ has non-zero $\left(\U\left(V_0^\perp\right), \mu\circ{\det}_{V_0^\perp} \right)$-period for some automorphic character $\mu$ of $E_1$, then the L-function
\[
  L\left(s,\pi\times\widetilde{\mu}^{-1}\right)
\]
is non-vanishing at $s=1/2$.
\end{Prop}
\vskip 5pt

\begin{proof}
The proof is similar to the proof of Theorem \ref{T:UFJ-1}. Let $f\in\pi$ be such that $\mathcal{P}_{V_0^\perp}(f,\mu)\neq 0$. Since $\pi$ is unitary, by taking complex conjugation, we see that
\[
  \mathcal{P}_{V_0^\perp}\left(\overline{f},\mu^{-1}\right)\neq 0.
\]
Let $W$ be the maximally split $2n$-dimensional skew-Hermitian space over $E$, and set 
\[  \chi_V=\chi_W=\widetilde{\mu}.\]
Consider the theta correspondence between $\U(V)$ and $\U(W)$ with respect to the additive character $\psi$ and auxiliary characters $(\chi_V,\chi_W)$. Applying Proposition \ref{P:G-SFJ}, we deduce that
\[
  \mathcal{F}_{\psi_B}\left(\theta(\varphi,f)\right)\neq 0
\]
for some $\varphi\in\mathcal{S}\left((Y\otimes V)(\A)\right)$. In particular, $\Theta(\pi)\neq 0$. Since $\pi_v$ is generic, by the fact that the local theta lift of $\pi_v$ to $\U(W_v)$ is the first occurrence, it is not hard to see that $\Theta(\pi)$ is the first occurrence of $\pi$ in the Witt tower containing $W$. Then by the Rallis inner product formula \cite[Lemma 10.2]{Ya}, the L-function $L\left(s,\pi\times\widetilde{\mu}^{-1}\right)$ is non-vanishing at $s=1/2$.

\end{proof}
\vskip 5pt

Notice that if $\pi$ is tempered cuspidal, then for any finite place $v$ of $F$ split in $E$, $\pi_v$ is a tempered representation of $\U(V_v)\cong\GL_{2n}(F_v)$, hence is generic. In fact, according to \cite{Ar} and \cite{At}, for almost all place $v$ of $F$, $\pi_v$ is generic.

\vskip 5pt

\begin{Cor}  \label{C:XZW}
In the context of Conjecture \ref{C:X-Z}, let $\pi$ be an irreducible tempered cuspidal representation of $\U(V)$ admitting non-zero $(\mu_1,\mu_2)$-linear period.
Then assertions (1) and (3) in Conjecture \ref{C:X-Z} hold. 
If we further assume that:
\begin{itemize}
	\item there is a finite place $w$ of $F$ split in $E$, such that $\pi_w$ is a discrete series of $\U(V_w)\cong\GL_{2n}(F_w)$;
	\item there exists some automorphic character $\xi$ of $E^\times$, such that $\widetilde{\mu}_1\widetilde{\mu}_2=\xi^2$.
\end{itemize}
Then assertion (2)  in Conjecture \ref{C:X-Z} also holds. 
\end{Cor}

\vskip 5pt

\begin{proof}
The assertion (1) in Conjecture \ref{C:X-Z} follows from Proposition \ref{P:XZ-1} directly, and the assertion (3) is trivial. To show the assertion (2) under the extra hypotheses, one can use the same argument as contained in \cite[\S 6]{PWZ}. See also \cite[Proof of Thm. 5.3]{GR} for another instance of this argument.

\end{proof}
\vskip 5pt

In the odd dimensional case we also have an analog of Proposition \ref{P:XZ-1}. 

\begin{Prop} \label{P:XZ-2}
Let $V$ be a $(2n+1)$-dimensional Hermitian space and $V_0$ be an $n$-dimensional nondegenerate subspace of $V$. Let $\pi$ be an irreducible cuspidal representation of $\U(V)$ such that $\pi_v$ is generic at some finite place $v$ of $F$.  If $\pi$ has non-zero $\left(\U\left(V_0^\perp\right), \mu\circ{\det}_{V_0^\perp} \right)$-period for some automorphic character $\mu$ of $E_1$, then the L-function
\[
  L\left(s,\pi\times\widetilde{\mu}^{-1}\right)
\]
has a pole at $s=1$.
\end{Prop}
\vskip 5pt 

\begin{proof}
Similar to the proof of Proposition \ref{P:XZ-1}.
\end{proof}

\vskip 5pt

\begin{Rmk}
This proposition gives an analog of \cite[Cor. 1.3(2)]{PWZ} for unitary groups. Indeed, our method is also applicable to:
\begin{itemize}
  \item the $\left(\SO_{2n+1} , \SO_n \times \SO_{n+1}\right)$-period integral, by considering the theta correspondence between $\SO_{2n+1} \times \Mp_{2n}$;
  \vskip 5pt

  \item the $\left(\SO_{2n} , \SO_{n-1} \times \SO_{n+1}\right)$-period integral, by considering the theta correspondence between $\SO_{2n} \times \Sp_{2n-2}$.
\end{itemize}
In these two cases, our method recovers \cite[Cor. 1.3(1) \& Cor. 1.3(2)]{PWZ} under slightly weaker hypothesis on the cuspidal representation $\pi$.
\end{Rmk}

\vskip 10pt

\section{\bf Unitary Bump-Friedberg and Jacquet-Shalika} \label{S:final}
In this final and somewhat speculative section, we would like to suggest the consideration of two variants of the branching problems studied above. They are motivated by  the proof of the global Theorem \ref{T:split-global}. These global results were direct consequences of two Rankin-Selberg integrals for the exterior square L-functions, due to Jacquet-Shalika \cite{JS} and Bump-Friedberg \cite{BF} respectively.  
\vskip 5pt

More precisely, the global zeta integral of Jacquet-Shalika \cite{JS} considers the generalized Shalika period $Sha(\Pi, \sigma)$, where one takes for $\sigma$ a mirabolic Eisenstein series on $\GL_n$. Likewise, the global zeta integral of Bump-Friedberg \cite{BF} considers the generalized linear period $Lin(\pi, \sigma \boxtimes \chi)$, where $\sigma$ is a mirabolic Eisenstein series. The resulting   unfolding of these global zeta integrals indicates that these period spaces are generically of multiplicity one. What we would like to propose here is the study of the analog of these period spaces in the unitary setting.
\vskip 5pt

Such an analog in the unitary setting was in fact studied in a paper of Furusawa-Morimoto \cite{FM} for the generalized Shalika period on $\U_4$. In that case, the representation $\sigma$ is a representation of $\U_2$, and one takes $\sigma$ to be an Eisenstein series (when the $\U_2$ is quasi-split). However, it is well-known that 
there is no  mirabolic Eisenstein series for general $\U_n$ beyond this case, since there is no analog of the mirabolic subgroup.
\vskip 5pt

We will propose what we feel is the appropriate conjectural extension here.
The main observation is that the mirabolic Eisenstein series for $\GL(V)$ is built out of Schwarz functions on $V_{\A}$. Locally, the natural representation of $\GL(V)$ on $\mathcal{S}(V)$ is nothing but the Weil representation of the dual pair $\GL(V) \times \GL_1$. Hence, the analog of the mirabolic Eisenstein series  in the unitary setting should be a Weil representation of $\U_n$. 

\vskip 5pt

With this in mind, let us begin in the local setting, with $E/F$ a quadratic extension of local fields. Fix a nontrivial additive character $\psi$ of $F$ and a conjugate-symplectic character $\mu$ of $E^{\times}$, so that $\mu|_{F^{\times}} = \omega_{E/F}$. Then for a skew-Hermitian space $W$, one  has an associated Weil representation $\omega_{W, \mu, \psi}$.

\vskip 5pt

Suppose now that $V$ is a maximally split Hermitian space of even dimension, with $V = X \oplus Y$ a Witt decomposition and $P(X)$ the Siegel parabolic stabilizing $X$. 
Elements in the unipotent radical $N(Y)$ of $P(Y)$ can be regarded as skew-Hermitian forms on $X$. For such a nondegenerate element $B \in N(Y)$, with associated Shalika subgroup $S_B = \U(X,B) \ltimes N(X)$, one may consider the Hom space
\[  Sha_B(\pi, \omega_{B, \mu, \psi}) :=  \Hom_{S_B}\left( \pi, \omega_{B, \mu, \psi} \boxtimes \psi_B\right) \]
for $\pi \in {\rm Irr}(\U(V))$.
\vskip 5pt

Likewise, when $W$ is a skew-Hermitian space with a nondegenerate decomposition $W = W_0 \oplus W_0^{\perp}$ such that $\dim W_0 = \dim W_0^{\perp}$, one may consider the Hom space
\[   Lin_{W_0}\left( \pi, \omega_{W_0, \mu,\psi}  \right) :=  \Hom_{\U(W_0) \times \U(W_0^{\perp})} \left( \pi, \,  \omega_{W_0, \mu,\psi} \boxtimes 1_{\U(W_0^{\perp})}\right) \]
for $\pi \in {\rm Irr}(\U(V)$. 
\vskip 5pt

Motivated by the split case, one can ask if these Hom spaces are of dimension $\leq 1$, and if one could determine their dimension precisely, as $\pi$ varies over a generic L-packet, for example. One can also consider the global analog of these local Hom spaces, namely the relevant global period integrals. Again, one can ask if these period integrals are given by special values of the exterior square L-function for the relevant unitary group. We have not given much thought to these questions  but feel that they may be worthwhile to investigate. Indeed, they are analogous to the twisted  GGP conjecture proposed in the recent paper \cite{GGP}.

\vskip 15pt
\noindent{\bf Acknowledgments:} This paper was initiated after listening to Wei Zhang's talk in the conference ``Relative aspects of the Langlands program, L-function and Beyond Endoscopy" in May 2021, during which he highlighted Conjecture \ref{C:X-Z}. We thank him for the initial inspiration and subsequent discussions on this conjecture. We also thank Chen Wan and Raphael Beuzart-Plessis for several email exchanges concerning their results on the unitary Shalika period, and are grateful to Jayce Getz for pointing out the relevance of  \cite{PWZ} and his paper \cite{GW} with Wambach. The work for this paper is partially supported by a Singapore government MOE Tier 1 grant R-146-000-320-114.
\vskip 5pt

Finally, it is a pleasure and privilege to be able to dedicate this paper to Steve Kudla on the occasion of his 70th birthday. Our own mathematical taste and professional work have been significantly influenced by  Steve's foundational work on the topic of  theta correspondence as well as his lucid mathematical writings. As the reader will observe, the main technique used in this paper is the transfer of periods by theta correspondence, whose proof is modelled on Steve's computation of the Jacquet modules of the Weil representation, given in his influential paper \cite{Ku} (which is his most highly cited paper on Mathscinet to date).
As such, we hope this paper is an appropriate contribution to his 70th birthday volume.

\end{document}